\documentclass[11pt,reqno]{amsart}
\usepackage{euscript,amscd,amsgen,amsfonts,amssymb,latexsym}

\newcommand{\eqdef}{\stackrel{\scriptscriptstyle\rm def}{=}}
\usepackage{bbm}
\usepackage{pgfplots, multicol}
\usepackage{mathtools}
\usepackage{paralist}
\usepackage{todonotes}
\usepackage{enumerate}
\usepackage{graphicx}
\usepackage{tikz-cd}
\usepackage[margin=1in]{geometry}
\usepackage{soul}
\usepackage[sort,nocompress]{cite}
\usepackage{xcolor}
\usepackage[normalem]{ulem}

\newtheorem{theorem}{Theorem}
\newtheorem{corollary}[theorem]{Corollary}
\newtheorem{proposition}[theorem]{Proposition}
\newtheorem{lemma}[theorem]{Lemma}

\newtheorem{question}[theorem]{Question}

\newtheorem{definition}[theorem]{Definition}

\newtheorem{example}[theorem]{Example}

\newtheorem*{thmA}{Theorem A}
\newtheorem*{thmB}{Theorem B}

\graphicspath{ {images/} }

\newcommand{\beha}{\begin{enumerate}}
\newcommand{\behe}{\end{enumerate}}
\renewcommand{\epsilon}{\varepsilon}

\newcommand{\Per}{{\rm Per}}

\newcommand{\Or}{\mathcal{O}}

\newcommand{\cM}{\mathcal{M}}

\newcommand{\bR}{{\mathbb R}}

\newcommand{\bZ}{{\mathbb Z}}
\newcommand{\bN}{{\mathbb N}}
\newcommand{\bQ}{{\mathbb Q}}

\newcommand{\cF}{{\mathcal F}}
\newcommand{\cW}{{\mathcal W}}
\newcommand{\cG}{{\mathcal G}}
\newcommand{\cA}{{\mathcal A}}

\newcommand{\cL}{{\mathcal L}}

\newcommand{\Ptop}{P_{\rm top}} 
\newcommand{\cS}{{\mathcal S}}

\DeclareMathSymbol{\varnothing}{\mathord}{AMSb}{"3F}
\renewcommand{\emptyset}{\varnothing}

\DeclareMathOperator{\ClosedShift}{\Sigma_{\text{\rm invariant}}}
\newcommand{\Gen}{\mathcal{G}}

\newcommand{\calA}{{\mathcal{A}}}

\newcommand{\calF}{{\mathcal{F}}}
\newcommand{\calL}{{\mathcal{L}}}
\newcommand{\calM}{{\mathcal{M}}}
\newcommand{\calO}{{\mathcal{O}}}

\def\1{1\!\!1}

\def\and{\text{ and }}

           \def\htop{h_{{\rm top}}}

         \def\P{\text{{\rm P}}}

\def\ES{{\rm ES}}

                        \def\^{\tilde}

\def\Per{{\rm Per}}

\def\var{{\rm var}}

\def\Per{{\rm Per}}

\def\1{1\!\!1}

\title[]{Computability of topological pressure on compact shift spaces beyond finite type}
\author{Michael Burr}
\author{Suddhasattwa Das}
\author{Christian Wolf}
\author{Yun Yang}
\address{Michael Burr, School of Mathematical and Statistical Sciences, Clemson University}
\email{burr2@clemson.edu}
\address{Suddhasattwa Das, Courant Institute of Mathematical Sciences, New York University and Department of Mathematics, George Mason University}
\email{dass@cims.nyu.edu}
\address{Christian Wolf, Department of Mathematics, City College of New York}
\email{cwolf@ccny.cuny.edu}
\address{Yun Yang, Department of Mathematics, Virginia Polytechnic Institute and State University}
\email{yunyang@vt.edu}

\thanks{Burr was partially supported by grants from the National Science Foundation (CCF-1527193 and DMS-1913119).}
\thanks{Das was partially supported by ONR MURI grant N00014-19-1-242 and ONR YIP grant N00014-16-1-264.}
\thanks{Wolf was partially supported by  grants from  the Simons Foundation (\#637594 to Christian Wolf) and PSC-CUNY (TRADB-51-63715 to Christian Wolf).}
\thanks{Yang was partially supported by a grant from the National Science Foundation (DMS-2000167).}
\begin{document}


\begin{abstract}
We investigate the computability (in the sense of computable analysis) of the topological pressure $\Ptop(\phi)$ on compact shift spaces $X$ for continuous potentials $\phi:X\to\bR$.
This question has recently been  studied for subshifts of finite type (SFTs) and their factors (Sofic shifts).
We develop a framework to address the computability of the topological pressure on general shift spaces and  apply this framework to  coded shifts. In particular, we prove the computability of the topological pressure for all continuous potentials on S-gap shifts,  generalized gap shifts, and particular Beta-shifts. We also construct shift spaces which, depending on the potential,  exhibit  computability and non-computability of the topological pressure.  We further prove  that the generalized pressure function $(X,\phi)\mapsto \Ptop(X,\phi\vert_{X})$ is not computable for a large set of shift spaces $X$ and potentials $\phi$.  In particular, the entropy map $X\mapsto h_{\rm top}(X)$ is computable at a shift space $X$ if and only if $X$ has zero topological entropy.
Along the way of developing these computability results, we derive several ergodic-theoretical properties of coded shifts which are of independent interest beyond the realm of computability.

\end{abstract}

\maketitle

\section{Introduction}
\subsection{Motivation}
The topological pressure  is a functional acting on the space of continuous potentials  $C(X,\bR)$ of a  dynamical system $f:X\to X$.  The pressure is a natural generalization of the topological entropy and encodes several  properties of the underlying dynamical system.  
Moreover, the pressure is one of the main components of the thermodynamic formalism, which, in turn, has played a key role in the development of the theory of dynamical systems. 
 Additionally, the variational principle for the topological pressure connects topological and measure-theoretic dynamics in a natural way.  This connection is often exploited in applications of the topological pressure including in the study of Lyapunov exponents, fractal dimensions, multi-fractal spectra, natural invariant measures (e.g., maximal entropy, physical, SRB, zero-temperature, and maximal dimension measures), and rotation sets, see \cite{Bowen,Keller1998,MauldinEtAl2003,Pesin1997,Ruelle2004} and the references therein. 
In addition, the pressure has seen applications in both mathematics and related fields, such as statistical physics and mathematical biology, see \cite{ADG,Demetrius2013,DG,DW1,Georgii2011,TS} and the references therein. For an overview of entropy and pressure, we refer the reader to the books and survey articles \cite{Bowen,Katok2007,MauldinEtAl2003,Ruelle2004,Walters1973}.

Computability theory determines the feasibility of computational experiments by analyzing the possible precision of the output when using approximate data.  Without a precision guarantee, a computer experiment might miss or misinterpret interesting behaviors.  The main idea behind computability theory is to represent mathematical objects, e.g., points, sets, and functions, by convergent sequences produced by Turing machines (computer algorithms for our purposes).  We say that a point, set, or function is computable if there exists a Turing machine that produces an approximation to any prescribed precision.  Using convergent sequences of points instead of single points allows for the study of the behavior of a larger class of objects by increasing the precision of an approximation, as needed, to adjust for sensitivity to the initial conditions.  For more thorough introductions to computability theory, see, \cite{BBRY,BHW, BY,BSW,GHR,RW,RYSurvey,K}.

Computability in dynamical systems has been the subject of intensive study during the last 20 years. In particular, the study of the computability of Julia sets, see, e.g., \cite{BBRY, BBY1, Br1, BY3, BY,  BY2, D1, DY}, and the computability of certain natural invariant measures, see \cite{BBRY,GHR,JP1,JP2}, has been addressed by many researchers.

Since the input to a computer algorithm must be finite, but descriptions of dynamical systems and potentials may require an infinite amount of data, algorithms may only use approximations given by a finite subset of the defining data.
 Therefore, it is not immediately clear whether it is possible for a computer to approximate the pressure of a given continuous potential on a dynamical system accurately.  Computability theory provides an approach to address the question of whether a finite amount of input data is sufficient to compute an approximation of the pressure.  We provide both necessary conditions as well as sufficient conditions to be able to approximate the pressure of a potential defined on a shift space.

Recently, there has been progress towards  understanding  the computability properties of the entropy of shift spaces, see, e.g., \cite{lind1984entropies,SpandlHertling_lang_2008,milnor2002entropy,Simonsen2006,Spandl2007,BSW,GHRS:Computability}.  
We refer the reader to \cite{GHRS:Computability,milnor2002entropy} for discussions about the computability of the topological entropy of general (non-symbolic) classes of dynamical systems.
The only classes of shift spaces for which the computability of the topological pressure is fully understood are subshifts of finite type (SFTs) and Sofic shifts \cite{Spandl_sofic_2008} (see also \cite{BSW}).  These two classes of subshifts share the feature that they can be completely described through finite combinatorial data, e.g., a transition matrix or a finite directed labeled graph.  
Hertling and Spandl \cite{hertling2008computability} and Spandl \cite{Spandl2007} proved that the language of an S-gap shift is decidable if and only if the set $S$ is computable. Recently, Hochman and Meyerovitch \cite{hochman2010characterization} proved that a real number is the entropy of a multidimensional SFT or Sofic shift if and only if it is right recursively enumerable.

Our goal is to establish computability results for the topological pressure for shift spaces which are described through countable, but not finite, data.  It is straight-forward to prove that once the language of a shift space is given by an oracle (or a Turing machine), then the topological pressure can be approximated from above by a computer, i.e., is upper semi-computable.  Our approach is  to prove that the topological pressure can be approximated from below by considering a sequence $\{X_m\}_{m\in\bN}$ of Sofic shifts which ``exhausts" $X$ in the sense that the topological pressure of $\phi$ restricted to these $X_m$'s converges to $\Ptop(\phi)$ from below.  This computation is nontrivial because the pressure does not vary continuously as the shift space changes. 
In the special case of the zero potential, the existence of such an approximation implies that $f$ is an almost Sofic shift, see \cite{Petersen1986}.
We provide sufficient conditions for potentials so that this sequence of approximations converges from below to the pressure.  We note that, in most cases, it is not sufficient to only have access to the language of the shift  space to be able to compute such an approximation, see Theorem B.  In order to compute the pressure, it is, therefore, necessary to provide additional information about the shift space.  We study the case where this additional information is given by a set of generators, i.e., $X$ is a coded shift.  

To emphasize the generality of the class of coded shifts, we recall that every shift space $X$ can be expressed as the closure of the set of
bi-infinite paths on a countable directed labeled graph. Such a graph is called  a representation of $X$, see \cite{MarcusLind1995}.
Coded shifts are precisely those shift spaces which have a representation by an irreducible countable directed labeled graph.
We remark that many classes of shifts spaces are coded shifts.  These classes include S-gap and generalized gap shifts, Sofic shifts, and Beta-shifts.  
To the best of our knowledge, generalized gap shifts are a new concept in symbolic dynamics.
They have the potential to provide examples which exhibit new dynamical phenomena, but their full study 
is beyond the scope of this paper.

We develop a complete computability theory for all coded shifts in Theorem A. Moreover, in Theorem B, we show that solely based on the language of shift spaces, the topological pressure is not computable for most potentials. To derive these results, we establish a general framework which can be applied to determine the computability of the topological pressure for other classes of shift spaces.  We observe that, in general, neither the generators nor the language of a coded shift can be derived from the other, so both are necessary to determine computability of the pressure.  

\subsection{Main results}

A two-sided shift space $X$ over a finite alphabet $\cA=\{0,\dots,d-1\}$ is {\em coded} if there exists a {\em generating set} $\cG=\{g_i\}_{i\in\bN}$ which is a subset of the set of all finite words, denoted by $\cA^\ast$, such that $X=X(\cG)$ is the smallest shift space that contains all bi-infinite concatenations of generators, i.e.,
\[ X_{\rm seq}\eqdef\{ \cdots g_{i_{-2}}g_{i_{-1}}g_{i_0}g_{i_1}g_{i_2}\cdots: i_j\in \bN, j\in \bZ\}.\]
In other words, $X$ is the topological closure of $X_{\text{seq}}$, for more details, see \cite{BlanchardHansel1986,FieF,MarcusLind1995}. 
 Let $X_m=X(\{g_1,\dots,g_m\})$ be the coded shift generated by the first $m$ generators in $\cG$.  
For a potential $\phi\in C(X,\bR)$, we denote the topological pressure by $\Ptop(\phi)=\Ptop(X,\phi)$. We refer the reader to Section \ref{sec:pressure+variation} for additional details.  Since each $X_m$ is a Sofic shift, the topological pressure $\Ptop\left(X_m,\phi|_{X_m}\right)$ is computable, see \cite{Spandl_sofic_2008}.  
We observe that $\{X_m\}_{m\in\bN}$ is an increasing sequence of shift spaces in $X_{\rm seq}$. In order to establish computability results for the topological pressure, we study when $\Ptop\left(X_m,\phi|_{X_m}\right)$ converges to $\Ptop(\phi)$.  For this convergence to hold, we show that the pressure on $X$ cannot be ``concentrated" on $X\setminus \bigcup_m X_m$.  To this end, we introduce the following definitions:

\begin{definition}\label{def:FSP}
Given a potential $\phi\in C(X,\bR)$, we say that the pair $(X,\phi)$ has {\em full sequential pressure} if
$$\Ptop(\phi)=\sup_{\mu\in \cM_X}\{P_\mu(\phi): \mu(X_{\rm seq})=1\},$$
where $P_\mu(\phi)=h_\mu(f)+\int \phi\, d\mu$ and $\cM_X$ is the set of $f$-invariant Borel probability measures on $X$.
We denote the set of potentials on $X$ with full sequential pressure by $FSP(X)$. 
Moreover, we say $X$ has {\em full sequential entropy} if the zero potential has full sequential pressure.
\end{definition}

\begin{definition}\label{def:FSSP}
Given a potential $\phi\in C(X,\bR)$, we say that the pair $(X,\phi)$ has {\em strict full sequential pressure} if 
$$
\sup_{\mu\in \cM_X}\{P_\mu(\phi): \mu(X_{\rm seq})=1\}>\sup_{\mu\in \cM_X}\{P_\mu(\phi): \mu(X\setminus X_{\rm seq})=1\}.
$$
We denote the closure of the set of potentials on $X$ with strict full sequential pressure by $FSSP(X)$ which we call the 
set potentials with {\em full sigma-sequential pressure}.
\end{definition}

From the properties of the pressure on shift-invariant subsets, see Section \ref{sec:pressure+variation}, we observe that $FSSP(X)\subset FSP(X)$.  In Theorem A, we show that, under suitable conditions, pressures of potentials in $FSSP(X)$ can be computed while whose in the complement of $FSP(X)$ cannot be computed.  We conjecture that the equality $FSP(X)=FSSP(X)$ holds for all coded shifts.

We use three notions of computability of functions: computable functions, semi-computable functions, and computability at a point in the domain.  Given a set $S\subset C(X,\bR)$, we say that a function $g:S\to\bR$ is computable if, for any input function $\phi\in S$, $g(\phi)$ can be calculated to any prescribed precision.  Semi-computability is a weaker notion of computability corresponding to one-sided convergence.  Additionally, computability at a point is a computable version of being continuous at a point.  We refer the reader to Section 3 and \cite{BurrWolf2018,GHR} for the precise definitions and details.

We show in Lemma \ref{lem:pres_upper} that once the language (i.e., the set of admissible words of $X$) of a shift space is given, the topological pressure function is upper semi-computable on $C(X,\bR)$.  Therefore, in order to classify the computability of the topological pressure function, it suffices to prove that this function is also lower semi-computable. In other words, we show the convergence of the topological pressure for $X_m$ identified above.  

We develop an ergodic-theoretic approach based on inducing to obtain a classification of the computability of the topological pressure.  In order to apply this technique, we use the following technical condition:
\begin{definition}[cf \cite{Pavlov:UniqueRepresentability}\footnote{We note that the definitions of unique decipherability and unique decomposition are reversed in the arXiv version \cite{pavlov} when compared to the journal version \cite{Pavlov:UniqueRepresentability}.  Our definition matches that of the journal version.}]
A coded shift $X$ is {\em uniquely representable} (also known  as {\em uniquely decomposable} or {\em unambiguously coded}) if there exists a generating set $\cG$ such that each $x\in X_{\rm seq}$ can be uniquely written
 as an infinite concatenation of elements in $\cG$. In this case, we say that $\cG$ is a {\em unique representation} of $X$.
\end{definition}  
The notion of uniquely representable differs from the weaker notion of uniquely decipherable \cite{BlanchardHansel1986} because the uniqueness there is required only for finite words.  We also note that the minimality of a generating set is a distinct concept from a unique representation.  The manuscript \cite{BPR21} includes a proof that all coded shifts are uniquely representable.

We prove the computability of the topological pressure as a function of the potential $\phi$ on a coded shift $X$ given by its generators and language.  We denote the set of all coded shifts that are contained in the bi-infinite full shift  $\Sigma$ by $\Sigma_{\rm{coded}}$.

\begin{thmA} 
Let $X$ be a coded shift with unique representation $\cG$.
\begin{enumerate}[\rm(i)]
\item If $\phi\in FSSP(X)$,  then $\Ptop(\phi)= \lim_{m\to\infty} \Ptop(X_m,\phi);$
\item Suppose that the language and generating set $\cG$ of $X$ are given by oracles.  Then there is a universal Turing machine, i.e., does not depend on $X$, that computes $\phi\mapsto \Ptop(\phi)$ for $\phi\in FSSP(X)$; and 
\item If $\phi\in C(\Sigma,\bR)$ with $\phi|_X\not\in FSP(X)$, then $\Ptop(\cdot,\phi)$ is not computable at $X$ in $\Sigma_{\rm{coded}}$.
\end{enumerate}
\end{thmA}

Since the topological pressure depends continuously on the potential but is not continuous as a function of the shift space, the main challenge in the proof of Statement (ii) of Theorem A is to show that the sequence of entropies or pressures for $X_m$  converges to that of $X$.  In the proof of Statement (iii) of Theorem A, we show that if the potential $\phi$ does not belong to $FSP(X)$, then the pressure of  the potential restricted to the subshifts $X_m$ does not converge to the pressure of $\phi$ on $X$.

We note that the computability of the topological pressure on $FSSP(X)$ is a stronger statement than the computability of the pressure at each potential in $FSSP(X)$. Indeed, the latter means that for 
any potential $\phi\in FSSP(X)$, there exists a Turing machine (possibly depending on $\phi$) which computes $\Ptop(\phi)$, while the former states that the existence of {\em one} Turing machine which uniformly computes  $\Ptop(\phi)$ for any $\phi \in FSSP(X)$.
We provide  a general criterion for $FSSP(X)=C(X,\bR)$  in Proposition \ref{proptriv},  and we show that single and generalized gap shifts and a class of Beta-shifts satisfy this criterion.  
In particular, for these shift spaces, the topological pressure is computable on the entire set of continuous potentials.

We note that even though Theorem A is stated for two-sided shift spaces, it can be extended to one-sided shift spaces $X^+$ when $X$ and $X^+$ have the same language.  For example, if $X^+$ is given by infinite paths in an irreducible countable directed labeled graph, then it is possible to construct a two-sided coded shift $X$ with the same language as $X^+$.  It then follows that the computability of the pressure on $X$ implies the computability of the pressure on $X^+$.  We apply this construction in Section 7 to establish the computability of the pressure for particular Beta-shifts.

We remark that \cite{BPR21} includes a construction for a unique representation of a coded shift.  We do not use this construction in the present paper since the sequential pressure properties of Definitions \ref{def:FSP} and \ref{def:FSSP} are not obvious for this generating set.  Instead, we construct specific unique representations for our examples in Section 7 for which the sequential pressure properties can be identified.

In Section \ref{sec:concl}, we construct examples of coded shifts with an open set of potentials for which the full sequential pressure property fails and the pressure is not computable.  Theorem A immediately leads to the following corollary for the computability of the entropy (see \cite{GHRS:Computability} for related results):

\begin{corollary}
Let $X$ be a coded shift with unique representation $\cG$.  Moreover, suppose that the language and generating set $\cG$ of $X$ are given by oracles.  If the zero potential on $X$ is in $FSSP(X)$, then the entropy of $X$ is computable.
\end{corollary}

We next prove that if the coded shift $X$ is replaced by a general shift space given by its language, the topological pressure is no longer computable in most cases.  We say that $X$ is a shift space if it is a closed (shift-)invariant subset of $ \Sigma$.  Let the set of all shift spaces (not merely coded shifts) that are contained in the bi-infinite full shift  $\Sigma$ be denoted by $\Sigma_{\rm invariant}$.  The main differences between Theorems A and B are that the set of shift spaces is enlarged and the given information is reduced, since general shifts are not coded.  Both of these extensions reduce the extent of the computability.

\begin{thmB} 
Let $X_0\in \Sigma_{\rm invariant}$  and $\phi_0\in C(\Sigma,\bR)$. Suppose that all equilibrium states of $\phi_0\vert_{X_0}$ have non-zero entropy, then the generalized pressure function $$P: \ClosedShift\times C(\Sigma, \mathbb{R})\rightarrow\mathbb{R},\quad (X, \phi)\mapsto  \Ptop(X,\phi\vert_X)$$ is not computable at $(X_0,\phi_0)$.
\end{thmB}

In Theorem B, an oracle of a shift space is a function that lists the language of $X$ in order of nondecreasing length.  This is in contrast with Theorem A, where oracles for both the language and the generating set are given as input.
The challenge in the proof of Theorem B is to identify shift spaces $Y_n$ which converge to $X_0$, but whose pressures $\Ptop(Y_n,\phi_0\vert_{Y_n})$ do not converge to $\Ptop(X_0,\phi_0\vert_{X_0})$. The key idea is to construct spaces $Y_n$ such that the nonwandering set of each $Y_n$ is a finite union of appropriately selected periodic orbits.  
 
We remark that the entropy assumption in Theorem B is equivalent to the condition that
\begin{equation}\label{eqposent}
\Ptop(X_0,\phi_0\vert_{X_0})>\sup_{\mu\in \mathcal{M}_{X_0}}\int \phi_0\, d\mu,
\end{equation}
see Lemma \ref{lem:posentropy}.  Characterizations of shifts which satisfy Inequality \eqref{eqposent} have been extensively studied in \cite{Climenhaga2019} where the potentials that satisfy this inequality are called {\em hyperbolic}.  This inequality 
implies that, for any positive entropy shift space $X_0$, the generalized pressure function is not computable for a nonempty  open set of continuous potentials.

The assumption that $\phi_0\vert_{X_0}$ has equilibrium states with non-zero entropy holds for many shift spaces and potentials.  For example, Theorem B implies the non-computability of the pressure for an open and dense set of continuous potentials on shift spaces with specification (including SFTs and Sofic shifts). More precisely, for shift spaces with specification, Inequality (\ref{eqposent}) holds for any H\"{o}lder continuous function \cite{Climenhaga}. Furthermore, for any positive entropy shift space and any potential $\phi_0$, there exists $-\infty\leq t_{\rm min}<0<t_{\rm max} \leq\infty$ such that Theorem B holds for all potentials $t\phi_0\vert_{X_0}$ with $t\in (t_{\rm min},t_{\rm max})$. In many cases, $t_{\rm min}=-\infty$ and $t_{\rm max}=\infty$. We note that the case $t_{\rm max}<\infty$ corresponds to a zero-entropy phase transition, i.e., when $t_{\rm max} \phi_0\vert_{X_0}$ reaches a zero-entropy maximizing measure.  The   analogous statement holds in the case $t_{\rm min}>-\infty$. 

We observe that the converse of Theorem B is not true. Indeed, in Example \ref{counter} of  Section \ref{sec:concl}, we give   an example of a shift space and potential with zero-entropy equilibrium states for which the pressure is not computable.

Finally, we obtain the following corollary for the computability of the topological entropy by applying Theorem B to the potential $\phi_0=0$:
\begin{corollary}\label{cor4}
The generalized entropy map $X\mapsto \htop(X)$ is computable at a nonempty shift space $X_0$ if and only if $X_0$ has zero entropy.  Moreover, there is one Turing machine that uniformly computes the topological entropy at all shift spaces with zero entropy.
\end{corollary} 

We remark that Corollary 5 uses the notion of computability at a point, see Section 3, and states more than that 0 is a computable number.

\subsection*{Outline} 
We review the dynamics on shift spaces in Section~\ref{sec:rev:shift}.  In Section~\ref{sec:rev:comput}, we review the basics of computability theory and begin its application to shift spaces.  We apply computability theory to SFTs and Sofic shifts in Section~\ref{sec:comput_fin} in preparation for our main theorems.  Theorems A and B are proved in Sections \ref{sec:proof:com} and \ref{sec:proof:counterexample}, respectively.  In Section~\ref{sec:applications}, we apply our main theorems to many  examples of classes of coded shifts, such as S-gap shifts, generalized S-gap shifts, a class of Beta-shifts, and Sofic shifts.  We end with some concluding remarks and open problems in Section~\ref{sec:concl}.

 \section{Preliminaries on the dynamics of shift spaces.} \label{sec:rev:shift}
 
We introduce the relevant background material from the theory of  symbolic dynamics  and the thermodynamic formalism.  

\subsection{Symbolic spaces}
We review some relevant material from symbolic dynamics, and we refer the reader to \cite{Kit,MarcusLind1995} for details.
Let $\Sigma=\Sigma^{\pm}_d$ denote the shift space of bi-infinite sequences $(x_k)_{k\in \bZ}$ in the finite alphabet $\cA=\cA_d=\{0,\dots,d-1\}$. Endowing $\Sigma$ with the {\em Tychonov product topology} makes $\Sigma$ into a compact  and metrizable topological space. In fact, for every  $0<\theta<1$, the metric given by
\begin{equation}\label{defmetX}
d(x,y)=d_\theta(x,y)\eqdef  \theta^{\min\{|k| \;:\;  x_k\neq y_k\}} 
\end{equation}
induces the Tychonov product topology on $\Sigma$.
The  {\em (left) shift map} $f:\Sigma \to \Sigma $ is defined by $f(x)_k=x_{k+1}$, and we note that $f$ is a homeomorphism. We call a closed shift-invariant set $X\subset \Sigma$ a {\em shift space}, and  we say that $f:X\to X$ is a {\em subshift}.

For $n\in\bN$, $\tau=\tau_0\tau_1\cdots\tau_{n-1}\in\cA^n$, and $i\in\bZ$, we define the {\em cylinder generated by $\tau$ starting at $i$} to be the set 
$$
[\tau]_i=\{x\in X:x_{i+j}=\tau_j\text{ for }j\in\{0,\dots,n-1\}\}.
$$
When $i=0$, we write $[\tau]\eqdef[\tau]_0$.  Similarly, for $x\in X$ and $i,j\in\bZ$ with $i\leq j$, we write $x[i,j]\eqdef x_i\cdots x_{j}$ for the {\em substring} of $x$ from $i$ to $j$.  In addition, we call $[x]_i^j\eqdef [x[i,j]]_i$ the {\em cylinder of length $j-i+1$ starting at $i$ generated by $x$}.  When $i=0$, we write $[x]^j$ for $[x]_0^j$.


For $n\geq 1$, we say that the $n$-tuple $\tau=\tau_0\tau_1\cdots \tau_{n-1}$ of elements in $\cA$ is  an {\em $X$-admissible word} provided $\tau$  occurs as a substring of one of the elements of $X$. The number $n$ is called the {\em length} of $\tau$ and is denoted by $|\tau|$.  We denote  the set of $X$-admissible words of length $n$ by $\cL(X,n)$.  We  call $\cL(X)\eqdef \bigcup_{n=0}^\infty \cL(X,n)$ the {\em language of $X$}. Given words $\tau$ and $\eta$ of lengths $n$ and $m$, respectively,
we denote  the word of length $n+m$ obtained by concatenating $\tau$ and $\eta$ by $\tau\eta$. 
Furthermore, we call  \[\Or(\tau)=\cdots \tau_0\cdots \tau_{n-1}.\tau_0\cdots \tau_{n-1}\tau_0\cdots \tau_{n-1}\cdots \in \Sigma\]  the {\em periodic point}   {\em generated} by $\tau$ (of {\em period} $n$). 
Here we use the notation of the ``decimal point" to separate the coordinates with nonnegative indices from those with negative indices.
We denote the set of all {\em periodic points} of $f:X\to X$ with {\em prime period} $n$ by $\Per_n(f)$, i.e., $n$ is the smallest positive integer such that $f^n(x)=x$. Moreover, $\Per(f)\eqdef \bigcup_{n\geq1} \Per_n(f)$ denotes the set of all periodic points of $f$. If $n=1$, then we say that $x$ is a {\em fixed point} of $f$. 
For $x\in \Per_n(f)$, we call $\tau_x=x_0\cdots x_{n-1}$ the {\em generating segment} of  $x$, that is, $x=\Or(\tau_x)$.

For $\phi\in C(X,\bR)$ and  $k\in\bN$, we define the {\em variation} of $\phi$ over cylinders of length $2k+1$ by
\[ \var_k(\phi)=\sup\{|\phi(x)-\phi(y)|: x_{-k}=y_{-k},\dots, x_{k}=y_{k}\}. \]
We denote  the set of all H\"older continuous functions $\phi\in C(X,\bR)$ with respect to the metric $d_\theta$ by $\cF_X$, i.e., $\cF_X$ consists of those functions such that there exists $C>0$ and $0<\alpha<1$ such that $\var_k(\phi)\leq C \alpha^k$ for all $k\in\bN$.

\subsection{Topological pressure and the variational principle}\label{sec:pressure+variation}
We review basic definitions from the thermodynamic formalism, see \cite{MauldinEtAl2003,Walters1973} for a more detailed account.
Let $\cM=\cM_X$ denote the set of all $f$-invariant Borel probability measures on $X$ endowed with the weak$^\ast$ topology, and let $\cM_{\rm erg}=\cM_{{\rm erg},X}\subset \cM$ be the subset of ergodic measures. We recall that $\cM$ is a convex metrizable topological space.  For $x\in \Per_n(f)$,  the {\em periodic point measure} of $x$ is given by $\mu_x=\frac{1}{n}(\delta_x+\dots +\delta_{f^{n-1}(x)}), $ where  $\delta_y$ denotes the  Dirac measure on $y$. We write $\cM_{\rm per}=\cM_{{\rm per},X}=\{\mu_x: x\in \Per(f)\}$, and we observe that $\cM_{\rm per}\subset \cM_{\rm erg}$.

Given $\mu\in \cM$,  the \emph{measure-theoretic entropy} of $\mu$  is given by
 \begin{equation*}
 h_\mu(f)=\lim_{n\to\infty}- \frac{1}{n}\sum_{\tau\in \cL(X,n)} \mu([\tau]^{n-1})\log (\mu([\tau]^{n-1})),
 \end{equation*}
 omitting terms with  $\mu([\tau]^{n-1})=0$.

 Let $f:X\to X$ be a subshift and $\phi\in C(X,\bR)$.  For $n\geq 1$, we define the {\em $n$-th partition function} $Z_n(\phi)$ at $\phi$ by
\begin{equation}\label{defnpart}
Z_n(\phi)=\sum_{\tau\in \cL(X,n)}\exp\left( \sup_{x\in [\tau]} S_n\phi (x) \right), 
\end{equation}
where  
\begin{equation}\label{eq:Sn}
S_n \phi(x) \eqdef \sum_{k=0}^{n-1} \phi(f^k(x)).
\end{equation}
We observe that the sequence $\left(\log Z_n(\phi)\right)_{n\geq 1}$ is subadditive, see, e.g., \cite{MauldinEtAl2003}. The \emph{topological pressure} of $\phi$ with respect to the shift map $f:X\to X$ is defined by 
\begin{equation} \label{eqn:def:P}
\Ptop(\phi)= \lim_{n\to \infty} \frac{1}{n} \log Z_n(\phi)=\inf \left\{\frac{1}{n} \log Z_n(\phi): n\geq 1\right\}.
\end{equation}
Moreover, $\htop(f)=\Ptop(0)$ denotes the {\em topological entropy} of $f$. If $Y\subset X$ is a subspace, we write $\Ptop(Y,\phi)$ for $\Ptop(f\vert_Y,\phi\vert_Y)$ and $\htop(Y)$ for $\htop(f\vert_Y)$.

 For $\phi\in C(X,\bR)$, we write $\mu(\phi)=\int \phi\, d\mu$. The quantity $P_\mu(\phi)=h_\mu(f)+\mu(\phi)$ is called the \emph{free energy} of $\mu$. The topological pressure satisfies the well-known \emph{variational principle}, see, e.g., \cite{Walters1973},

 \begin{equation}\label{varpri}
 P_{\rm top}(\phi)=\sup_{\mu\in \cM} h_\mu(f)+\mu(\phi).
 \end{equation}
The supremum on the right-hand side of Equation~\eqref{varpri} remains unchanged if $\cM$ is replaced by $\cM_{\rm erg}$.
If $\mu\in\cM$ achieves the supremum in Equation~\eqref{varpri}, then  $\mu$ is called an \emph{equilibrium state} of $\phi$. The set of equilibrium states of $\phi$ is denoted by $\ES(\phi)$. We recall that, since $f$ is expansive, the entropy map $\mu\mapsto h_\mu(f)$ is upper semi-continuous, which, in turn, implies that $\ES(\phi)$ is nonempty  \cite{Walters1973}. Furthermore, $\ES(\phi)$ is a compact and convex subset of $\cM$ whose extremal points are the ergodic equilibrium states in $\ES(\phi)$.
 

\section{Preliminaries on computability} \label{sec:rev:comput}

Computability theory identifies which quantities can and cannot be computed on a real-world computer.  
We model a computer as a Turing machine and say that a mathematical object (e.g., a point, set, or function)\ 
is computable if there exists a Turing machine which can approximate the object up to any desired precision.  A main idea of computability theory is that mathematical objects are replaced by sequences of approximations.  For our purposes, a Turing machine can be represented by an algorithm (using a bit-based model of computation), see, e.g., \cite{K}.

We review the basic definitions and results from computability theory.  For a more thorough introduction, see \cite{BBRY,BHW, BY,BSW,GHR,RW,RYSurvey,K}.  We use different, but closely related, definitions to those in \cite{BY} and \cite{GHR}, see also \cite{BSW}. 

\subsection{Computability theory for real numbers}

We present the basic definitions for the computability theory for real numbers.  Since many of our results use this theory, we review the details in these cases to make the later proofs clearer.  In computability theory, a real number is defined by a convergent sequence of rational numbers.

\begin{definition}
An {\em oracle approximation} of a real number $x\in\mathbb{R}$ is a function $\psi$ such that on input $n\in\mathbb{N}$, $\psi(n)\in\mathbb{Q}$ such that $|x-\psi(n)|<2^{-n}$.  A real number $x$ is said to be {\em computable} if there exists a Turing machine which is an oracle for $x$. 
\end{definition}

Loosely speaking, a computable number is one that can be algorithmically approximated to any requested precision.  We observe that the definition of a computable number not only includes an approximation to $x$, but, also, an explicit error estimate on the quality of the approximation.  In some cases, we can only compute a one-sided approximation to $x$, in particular, we cannot compute an error estimate.

\begin{definition}
A real number $x\in\mathbb{R}$ is {\em upper semi-computable} if there is a Turing machine $\psi$ such that
 the sequence $(\psi(n))_{n\in\bN}\in\bQ^\bN$ is non-increasing and converges to $x$.
Similarly, $x$ is {\em lower semi-computable} if $-x$ is upper semi-computable.
\end{definition}

We note that computability for a real number is equivalent to simultaneous upper and lower semi-computability.  
Computability theory extends beyond real numbers and can be adapted to describe other mathematical objects, such as functions. 

\begin{definition} 
Let $S\subset\mathbb{R}$.  A function $g:S\rightarrow\mathbb{R}$ is {\em computable} if there is a Turing machine $\chi$ so that, for any $x\in S$ and oracle $\psi$ for $x$, $\chi(\psi,n)$ is a rational number so that $|\chi(\psi,n)-g(x)|<2^{-n}$.  If, instead, there is a Turing machine $\chi$ such that for every $x$ and oracle $\psi$ for $x$, the sequence $(\chi(\psi,n))_{n\in\bN}$ is nonincreasing and converges to $g(x)$, then $g$ is called {\em upper semi-computable}.  A function $g$ is {\em lower semi-computable} if $-g$ is upper semi-computable.
\end{definition}

The conditions on $g$ to be computable or upper semi-computable imply that, for any fixed $\psi$, the function $n\mapsto\chi(\psi,n)$ is an oracle approximation or an approximation from above, respectively, of $g(x)$.
We also note that computable functions are continuous.  In particular, if the Turing machine $\chi$ queries $\psi$ up to precision $k$ when 
computing $\chi(\psi,n)$, then any $x'$ sufficiently close to $x$ has an oracle $\psi'$ that agrees with $\psi$ up to precision $k$. This implies that $\chi(\psi,n)=\chi(\psi',n)$ and $|g(x)-g(x')|<2^{-n+1}$.
Our main approach to show that a function is not computable is to show that it is not continuous, 
but we note that the converse is false, i.e., continuity does not imply computability.
 
\subsection{Computable metric spaces}

We provide the abstract theory for computable metric spaces for which the computability theory for real numbers is a special case.  We provide the minimal theory needed and refer the interested reader to \cite{BY,BSW,GHR} for more details and additional discussions.

\begin{definition}
Suppose that $(X,d_X)$ is a separable metric space and that $\mathcal{S}_X=(s_i)_{i\in\bN}$ is a dense sequence of points in $X$.  We say that $(X,d_X,\mathcal{S}_X)$ is a {\em computable metric space} if there is a Turing machine $\chi:\mathbb{N}^2\times\mathbb{Z}\rightarrow\bQ$ such that $|\chi(i,j,n)-d_X(s_i,s_j)|<2^{-n}$.
\end{definition}

Equivalent definitions for a computable metric space are to require either the function $(i,j)\mapsto d_X(s_i,s_j)$ to be a computable function or the distances between $s_i$ and $s_j$ to be uniformly computable in $i$ and $j$.  As an example, the triple $(\bR,d_{\bR},(q_i)_{i\in\bN})$ where $\bQ=\cup\{q_i\}$ is a computable metric space.  In this more abstract setting, we  generalize the concepts of oracles for  numbers and computable functions.

\begin{definition}
Suppose that $(X,d_X,\mathcal{S}_X)$ is a computable metric space.  An {\em oracle approximation} of $x\in X$ is a function $\psi$ such that on input $n$, $\psi(n)\in\bN$ such that $d_X(x,s_{\psi(n)})<2^{-n}$.  Moreover, $x$ is {\em computable} if there exists a Turing machine which is an oracle for $x$.
\end{definition}

The definitions for a computable function $f:S\rightarrow\bR$ for $S\subset X$ as well as upper and lower semi-computable functions are analogous to the definitions in the real case and we leave the details to the reader.  Many properties of interest, such as continuity, are retained in this definition.  We present an equivalent form of computability of a function that proves to be a useful tool in our proofs.  Although this result is well-known to experts, we include the proof as it is less common in the literature.  We break the result into the following lemma and proposition:

\begin{lemma}\label{lem:simulate}
Let $(X,d_X,\mathcal{S}_X)$ be a computable metric space and suppose that $S\subset X$.  In addition, suppose that the function $f:S\rightarrow\bR$ is computable with oracle Turing machine $\chi$.  There is a Turing machine $\phi$ such that on input $(i,k)\in\bN^2$, the output of the Turing machine is a (possibly infinite) sequence of pairs $(a_j,r_j)\in\bQ^2$ where $r_j>0$ such that for all $x\in B(s_i,r_j)$, $|f(x)-a_j|<2^{-k}$.  Moreover, for any oracle $\psi$ of $s_i$, $\chi(\psi,k)$ appears as some $a_j$ in this sequence.
\end{lemma}

\begin{proof}
We first construct a Turing machine to simulate all oracles for the point $s_i$.  In particular, we construct all $(2n+1)$-tuples of natural numbers $w=(w_{-n},\dots,w_n)$  such that for all $-n\leq l\leq n$, $|s_{w_{l}}-s_i|<2^{-{l}}$.  We call such a sequence a {\em partial oracle}.  We remark that for any oracle $\psi$ for $s_i$, $(\psi(-n),\dots,\psi(n))$ is a partial oracle, and, conversely, for any partial oracle $w$, there is an oracle $\psi$ for $s_i$ such that $w_{l}=\psi(l)$ for $-n\leq l\leq n$.

Every partial oracle can be constructed with the following procedure: For fixed $m\geq 0$, let $S_m=\{s_0,\dots,s_m\}$.  For each $0\leq {l}\leq m$, we compute an upper bound $d_l$ on $d_X(s_{l},s_i)$ with error at most $2^{-m}$.  For each sequence $w$ in $(S_m)^{2n+1}$, we accept $w$ if and only if $d_{w_l}<2^{-l}$ for all $-n\leq l\leq n$, i.e., $w$ forms a partial oracle.  By increasing $m$ and $n$, we construct all partial oracles.

We construct a new Turing machine to simulate $\chi(\psi,k)$ whose input is a partial oracle, but returns a failure if a precision outside the range of the partial oracle is requested by $\chi$.
This Turing machine uses a partial oracle $w$ in place of the oracle $\psi$ for $s_i$. Since a Turing machine only queries $\psi$ finitely many times, when the partial oracle is sufficiently large, we observe that the this simulation agrees with $\chi(\psi,k)$.

Finally, for any partial oracle $w$, we  construct a ball around $s_i$ of radius $r=\min_{|l|\leq n} (2^{-l}-d_l)$.  For every point in this ball, there exists an oracle $\psi'$ that agrees with $w$ on queries between $-n$ and $n$.  Therefore, if the simulated Turing machine for $f$ terminates on $w$, then the Turing machine also terminates on $\psi'$ and produces the same value.  Therefore, the estimate $a$ on $f(s_i)$ produced by $w$ satisfies $|f(x)-a|<2^{-k}$ in the ball $B(s_i,r)$.
\end{proof}

\begin{proposition}
Let $(X,d_X,\mathcal{S}_X)$ be a computable metric space and suppose that $S\subset X$.  A function $f:S\rightarrow\mathbb{R}$ is computable if and only if there exists a Turing machine $\chi$ such that, on input $(q,n)\in\bQ\times\bN$, $\chi(q,n)$ is a (possibly infinite) sequence of pairs of integers $(n_{q_j},m_{q_j})$ so that 
$$
f^{-1}(q-2^{-n},q+2^{-n})=\bigcup_j B\left(s_{n_{q_j}},2^{m_{q_j}}\right).
$$
\end{proposition}

\begin{proof}
Suppose that the conclusion is true and fix $x\in X$.  For any integer $n$, we show how to approximate $f(x)$ with error at most $2^{-n}$ as follows: We consider the collection of intervals $(m2^{-n}-2^{-n},m2^{-n}+2^{-n})$ for $m\in\bZ$, which form a countable cover of $\bR$.  For any fixed $k$, we  compute $k$ balls  in the preimages of those intervals with $|m|\leq k$.  For each such ball, we test if $x$ is in that ball by approximating the distance between $x$ and the center of the ball.  By increasing $k$ as well as the precision of this approximation, we eventually find a ball in the preimage of the interval for $m$ that contains $x$.  Therefore, the interval for $m$ contains $f(x)$ and $m2^{-n}$ is an approximation for $f(x)$ with error at most $2^{-n}$.

On the other hand, suppose that $f$ is computable.  Suppose that $y\in X$ with $f(y)\in(q-2^{-n},q+2^{-n})$.  We note that this containment can be detected by a Turing machine when $f(y)$ is is computed at high enough precision. Suppose that $y=s_i\in \cS_X$ and $\psi_i$ is an oracle for $s_i$.  We use the proof of Lemma \ref{lem:simulate} to construct a balls $B(s_i,r)$ from the oracle $\psi_i$ such that for all $x\in B(s_i,r)$, $f(x)\in (q-2^{-n},q+2^{-n})$.  The union of these balls for all $s_i$ with  $f(s_i)\in(q-2^{-n},q+2^{-n})$ is a subset of $f^{-1}(q-2^{-n},q+2^{-n})$, but, {\em a priori}, the containment may be proper.

We now show that these balls cover the preimage.  For any $x\in X$ with $f(x)\in(q-2^{-n},q+2^{-n})$, suppose that $\psi$ is an oracle for $x$.  Since the Turing machine for $f$ only queries $\psi$ finitely many times, there is an $s_i$ sufficiently close to $x$ such that there is an oracle $\psi_i$ that agrees with $\psi$ on the queries performed by the Turing machine.  Moreover, $s_i$ can be chosen sufficiently close to $x$ so that the disk around $s_i$, as constructed in Lemma \ref{lem:simulate}, is contains $x$ and is contained in $f^{-1}(q-2^{-n},q+2^{-n})$.  Thus, we achieve the desired set-equality.
\end{proof}

\subsection{Computability of subshifts}

We present the computability theory for subshifts.

\begin{definition}
Suppose that $X\in\ClosedShift$ with finite alphabet $\mathcal{A}_d$ and let $x\in X$.  An {\em oracle approximation} for $x$ is a function $\psi$ such that on input $n\in\bN$, $\psi(n)$ is the word $x[-n,n]$.  An {\em oracle approximation} for (the language of) $X$ is a function $\psi$ such that on input $n\in\bN$, $\psi(n)$ is a (finite) list of all admissible words of $X$ of length $2n+1$.  A point $x$ or a shift space $X$ is said to be {\em computable} if there is a Turing machine which is an oracle for $x$ or $X$, respectively.

Furthermore, an {\em oracle approximation for (the language of) $X$ from above} is a function $\psi$ such that on inputs $n,k\in\bN$, $\psi(n,k)$ is a finite list of words of length $2n+1$ which includes all admissible words of $X$ of length $2n+1$.  Moreover, for fixed $n$, the sequence of lists $(\psi(n,k))_{k\in\bN}$ is decreasing, and, for $k\in\bN$ sufficiently large, $\psi(n,k)$ is a list of all admissible words of $X$ of length $2n+1$.  The shift space $X$ is said to be {\em upper semi-compuable} if there is a Turing machine which is an oracle approximation for $X$ from above.  Similarly, we define a space $X$ to be {\em lower semi-computable} if its complement is {\em upper semi-computable}.
\end{definition}

We observe that the full shift $X=\Sigma$ is a computable shift since all words of length $2n+1$ can be explicitly listed by an algorithm.  For points in $\Sigma$, we use the metric 
$$d(x,y)=\begin{cases} 2^{-k} & \text{if } x\neq y \text{ and } k \text{ is the minimum value of $i$ such that } x[-i,i]\not=y[-i,i],\\
0 & \text{if } x=y
\end{cases}.$$
We note that this metric coincides with the $d_{1/2}$-metric  in Equation \eqref{defmetX}.
In addition, given oracles for $x$ and $y$, a Turing machine can compute $d(x,y)$ to any precision by comparing $x[-i,i]$ to $y[-i,i]$ for $i$ sufficiently large.  Therefore, we note that when $X$ is a computable metric space, the distance function is a computable function.

Upper and lower semi-computable subshifts are so named because the corresponding set $X$ is upper or lower semi-computable in $\Sigma^{\pm}$, see, e.g., \cite{BBRY,RYSurvey} for more details.  In the literature, computable subshifts are sometimes called decidable, and upper semi-computable subshifts are called effective \cite{DecidableSubshift:2019,EffectiveSubshift:2016}.

Less precisely, an approximation for $X$ from above gives a list of words that may appear in $X$, while an approximation for $X$ from below gives a list of words that must appear in $X$.  We observe that the computability for $X$ is equivalent to upper and lower semi-computability since, for $k$ sufficiently large, the upper and lower approximations are equal for fixed $n$.

In order to show that $X$ is a computable metric space, we fix the lexicographic total ordering on $\Sigma$.  In this ordering, $x<y$ if and only if either $x_0<y_0$ or there exists a $k$ such that $x[-k,k]=y[-k,k]$ and either $x_{k+1}<y_{k+1}$ or else $x_{k+1}=y_{k+1}$ and $x_{-k-1}<y_{-k-1}$.  We note that given oracles for distinct points $x$ and $y$, there exists a Turing machine that can decide which is greater.  Determining equality, however, is undecidable.

\begin{lemma}\label{lem:X:computable}
Let $X\in\ClosedShift$ with the finite alphabet $\mathcal{A}_d$.  Assume that $X$ is given by an oracle $\psi$.  Then $X$ is a computable metric space.
\end{lemma}

\begin{proof}
Since the distance between points given by oracles can be approximated to arbitrary precision, the key point in showing that $X$ is a computable metric space is to algorithmically construct a countable dense subset of points in $X$.  For each word $\tau\in\psi(n)=\cL(X,n)$ and index $i\in\bZ$, we choose $s_{\tau,i}$ to be the smallest element, lexicographically in $\cL(X)\cap[\tau]_i$.  This is a countable dense set since there are only countably many cylinders and for any $x\in X$, $s_{[x]_{-n}^n}$ is within $2^{-n}$ of $x$.

We next provide a Turing machine which is a uniform oracle for $s_{\tau,i}$.  In particular, we show how to compute $s_{\tau,i}[-k,k]$ for any $k\in\bN$.  Let $\ell=\max\{k,|i|,i+n\}$ and $\eta$ be the lexicographically smallest word in $\cL(X,\ell)\cap[\tau]_i$, i.e., $\eta_{i+j}=\tau_j$ for all $0\leq j\leq n-1$.  Thus, $s_{\tau,i}[-k,k]=\eta_{-k}\dots\eta_k$.
\end{proof}

In Section \ref{sec:comput_fin}, we show that the space $\ClosedShift$ is a computable metric space using the following distance function:  Suppose that $X,Y\in\ClosedShift$, then define
$$d(X,Y)=\begin{cases} 2^{-k} & \text{if } X\neq Y \text{ and } k \text{ is the minimum value of $i$ such that }\cL(X,i)\not=\cL(Y,i),\\
0 & \text{if } X=Y.
\end{cases}$$
It is straight-forward to show that this is a metric.  In addition, when $\psi$ and $\phi$ are oracles for $X$ and $Y$, respectively, the distance between two shifts can be algorithmically computed.  We leave the details to the reader.

We prove that the space of potentials on a shift space $X$ given by an oracle $\psi$ with the finite alphabet $\mathcal{A}_d$ is a computable metric space.  For additional details, we refer the reader to \cite{BurrWolf2018}.  We observe that the locally constant potentials with rational values (denoted by $LC(X,\bQ)$) are dense in $C(X,\bR)$  with respect to the supremum norm.  Since $LC_k(X,\bQ)$ is in bijective correspondence with $\mathbb{Q}^{|\psi(n)|}$, it follows that each potential in $LC(X,\bQ)$ can be represented by a pair $(k,q)$ where $q\in\mathbb{Q}^{|\psi(n)|}$.

\begin{definition}
Suppose that $\phi\in C(X,\bR)$.  An {\em oracle} for $\phi$ is a function $\chi$ such that on input $n$, $\chi(n)$ is a locally constant potential in $LC(X,\bQ)$ such that $\|\chi(n)-\phi\|_\infty<2^{-n}$.  Moreover, $\phi$ is {\em computable} if there is a Turing machine $\chi$ which is an oracle for $\phi$.
\end{definition}

Since an oracle $\psi$ for $X\in\ClosedShift$ can list all the cylinders of $X$ of any length $k$, there are Turing machines which list all potentials in $LC_k(X,\bQ)$.  Since the maximum of locally constant potentials with known cylinder length can be computed, the supremum of the difference between two locally constant potentials can be computed.  From this, it directly follows that $C(X,\bR)$ is a computable metric space.
 \section{Computability for subshifts presented by finite data} \label{sec:comput_fin}
 
We discuss the computability of SFTs and Sofic shifts,  which we call presented by finite data as they can be completely described by finite combinatorial data. Our key tool to proving computability for a coded shift $X$ is to approximate $X$ from inside by a sequence of subshifts presented by finite data.  We begin by recalling the definition of SFTs and Sofic shifts.  A subshift of finite type (SFT) is a subshift which can be described by a finite set of forbidden words. A Sofic shift is the collection of all edge sequences in a finite, directed, edge-labeled graph.
\begin{definition}[SFTs]
Let $\calF\subset\calA^\ast$ be a finite set of words, called the {\em forbidden words}.  The {\em subshift of finite type (SFT)} $X_\calF$ is the largest shift space where no element of $\calF$ appears as a subword of any $x\in X_\calF$.
\end{definition}

\begin{definition} [Sofic shifts] Let $\mathcal{T}=(G,E,L)$ be a labeled directed graph, where $G$ is a graph with directed edge set $E$ and the {\em labeling function} $L: E\rightarrow \mathcal{A}$ assigning a {\em label} $L(e)$ from the finite alphabet $\mathcal{A}$ to each edge $e\in E$. Let $\xi=\cdots e_{-1}e_0e_1\cdots$ be a bi-infinite path on $G$, i.e., $\xi$ is a point in the edge shift $X_{(G,E)}$.  The {\em label} of the path $\xi$ is
$$L(\xi)=\cdots L(e_{-1})L(e_0)L(e_1)\cdots\in \mathcal{A}^{\mathbb{Z}}.$$
The set of all bi-infinite labels of paths is denoted by  $X_{\mathcal{T}}=\{L(\xi): \xi\in G\}$.  
A subset $X$ of the full shift is called a {\em Sofic shift}  if $X=X_{\mathcal{T}}$ for some labeled graph $\mathcal{T}$.
\end{definition}

We refer the interested reader to \cite{MarcusLind1995} and references therein for additional details on SFTs and Sofic shifts.
Sofic shifts are those subshifts that are  factors of SFTs \cite[Section 3]{MarcusLind1995}. 
It is shown in \cite{BlanchardHansel1986} that every transitive Sofic shift is a coded shift.
We illustrate how the finite data defining SFTs and Sofic shifts can be used to algorithmically construct and study these shifts.

\begin{lemma} \label{lem:SFTcomputable}
Suppose that $\mathcal{F}$ is a finite collection of forbidden words that defines the SFT $X_{\cF}$.  There exists a Turing machine which takes $\mathcal{F}$ as input and produces a finite alphabet $\cA_{d'}$ and transition matrix $A'$ such that $X_{\cF}$ is conjugate to $X_{A'}$ via the map $h:X_{\cF}\rightarrow X_{A'}$.  Moreover, there exists a Turing machine that computes an oracle for $\phi_{\cF}\circ h^{-1}\in C(X_{A'},\bR)$ from any oracle for $\phi_{\cF}\in C(X_{\cF},\bR)$.  Conversely, there is a Turing machine that computes an oracle for $\phi_{A'}\circ h\in C(X_{\cF},\bR)$ from any oracle for $\phi_{A'}\in C(X_{A'},\bR)$.
\end{lemma}
\begin{proof}
Let $k$ be the length of the longest word in $\cF$.  If $k \leq 1$, then $X_{\cF}$ is a full shift, perhaps on a smaller alphabet. Therefore, we assume that $k>1$. Suppose that there are $d_{k-1}$ words of length $k-1$ which do not contain an element of $\cF$ as a subword.  Let $\{w_1,\dots,w_{d_{k-1}}\}$ be the set of these words.
Moreover, these words can be algorithmically computed by discarding words in $\cA_d^{k-1}$ which contain elements of $\cF$ as subwords.

Let $A$ be the transition matrix with $d_{k-1}$ rows and columns corresponding to the words $w_i$ where $A_{i,j}=1$ if and only if $w_i= x_0\cdots x_{k-1}$, $w_j=x_1\cdots x_k$, and $x_0 \cdots x_k$ is not an element $\cF$.
The matrix $A$ is trimmed by removing all words that can not appear in bi-infinite sequences, i.e., when it is impossible to extend these words in either the forward or the backward direction. More precisely, we iteratively remove all words $w_i$ with corresponding rows or columns consisting of all zeros in the transition matrix $A$.  When this procedure terminates, the resulting square matrix is the transition matrix $A'$ and the shift $X_{A'}$ is conjugate to $X_{\cF}$.  Moreover, $d'$ is the number of rows of $A'$.

For any $x_{A'}\in X_{A'}$, let $x_{\cF}$ be its image under the conjugacy.  We observe that $x_{A'}[i,j]$ determines $x_\cF[i,j+k-1]$ since every character in $x_{A'}$ corresponds to a word of length $k-1$ in the original alphabet.  Therefore, if $d(x_{A'},y_{A'})=2^{-n}$, then $d(x_{\cF},y_{\cF})\leq 2^{-n}$.  This correspondence between words in $X_{A'}$ and $X_{\cF}$ can be computed by a Turing machine since the transformation between them is a direct replacement.  Since the composition of computable maps is computable, the conjugated potentials are also computable.
\end{proof}

We use this lemma to prove that the space of shift spaces is a computable metric space.

\begin{corollary}
The space $\ClosedShift$ with finite alphabet $\mathcal{A}_d$ is a computable metric space.
\end{corollary}
\begin{proof}
As in the proof of Lemma \ref{lem:X:computable}, the challenge in this proof is to define a countable dense subset of points in $\ClosedShift$.  For any $S\subset \mathcal{A}^{2n+1}$, we define $X_S$ to be the SFT whose forbidden words are $S$, provided $X_S\not=\emptyset$.  By Lemma \ref{lem:SFTcomputable},  this shift space can be computed. Since $\mathcal{A}_d$ is finite, this collection of SFTs is a uniformly computable collection of points which is dense in $\ClosedShift$.
\end{proof}

We now discuss the computability properties of Sofic shifts.

\begin{lemma} \label{lem:shjbc0}
Let $X$ be a Sofic shift given by a labeled directed graph $\cG$.  There is a Turing machine that computes an SFT $X_A$ given by a transition matrix $A$ such that $X$ is a factor of $X_A$ via the factor map $h : X_A \to X$.  Moreover, there exists a Turing machine which produces an oracle  for $\phi_{A}\circ h\in C(X_A,\bR)$ from any oracle for $\phi\in C(X,\bR)$.
\end{lemma}

\begin{proof} 
Let $A$ be the \emph{edge-adjacency} matrix of $\cG$. The  SFT $X_A$ corresponding to $A$ is an extension of $X_{\cG} = X$. Since the graph $\cG$ is finite, $A$ can be computed from $\cG$. The construction of an oracle for $\phi_A$ is similar to the construction in Lemma~\ref{lem:SFTcomputable} and is left to the reader.
\end{proof}

We end this section with a discussion of renewal systems, which are coded shifts with a finite generating set.    We refer the interested reader to \cite{restivo1992note,restivo1989finitely} and the references therein for more details.  We observe that renewal shifts are Sofic.  We further note  that  renewal shifts are, in general, not SFTs.

\begin{lemma} \label{lem:jkc83k}
Let $\cW$ be a finite set of words in the alphabet $\cA_d$, and let $X(\cW)$ be the coded shift with generating set $\cW$.  There exists a Turing machine which takes $\cW$ as input and produces a labeled directed graph $\cG$ such that the Sofic shift $X_\cG$ is conjugate to $X(\cW)$.
\end{lemma}

\begin{proof}
The constructed labeled graph is a generalization of a bouquet of circles: There is one central vertex $v$ and each word $w\in\cW$ corresponds to a directed path of length $|w|$ beginning and ending at $v$ and labeled by the characters in $w$ in order.  Since $\cW$ consists of finitely many words, this graph is finite and can be constructed by a Turing machine.
\end{proof}

\section{Computability of the topological pressure for coded shifts}\label{sec:proof:com}

The main goal of this section is to prove  Theorem A.  First, we discuss the computability of the  topological pressure for some special classes of shift spaces and then prove that the pressure is, in general, computable from above.  The remainder of this section is devoted to the proof of the lower computability of the pressure for coded shifts and potentials satisfying the assumptions of Theorem A.

\subsection{Computability of the pressure in special cases}

When the shift space $X$ is presented by finite data, e.g., SFT, Sofic, or renewal, this additional information can be leveraged to derive computability.

\begin{proposition}\label{cpSofic}Suppose that $X$ is a Sofic shift given by a directed labeled graph $\cG$. Then the topological pressure of $X$ is computable.
\end{proposition}
\begin{proof}
In \cite{Spandl2007, Spandl_sofic_2008}, it is shown that the topological pressure of an SFT is computable from its transition matrix.  By Lemmas \ref{lem:SFTcomputable} and \ref{lem:shjbc0}, we compute an SFT $X_A$ such that $X$ is a finite-to-one factor of $X_A$ from the labeled graph $\cG$.  Since finite-to-one  factors preserve pressure \cite[Theorem 16]{Spandl_sofic_2008},  the topological pressure of $X_A$ and the topological pressure of $X$ are computable.
 \end{proof}

\begin{corollary}\label{cSofic}Let $\cW$ be a finite set of words in the alphabet $\cA_d$ and let $X(\cW)$ be the coded shift with generating set $\cW$. Then the topological pressure of $X$ is computable. 
\end{corollary}
\begin{proof}
From the code words $\cW$, we use Lemma \ref{lem:jkc83k} to compute the corresponding labeled directed graph $\cG$.  The result then follows from Proposition \ref{cpSofic}.
\end{proof}
 
\subsection{Upper semi-computability of the topological pressure}
We show that the topological pressure of a shift space $X$ is upper semi-computable when
 its language is given by an oracle. We note that this result holds for all shift spaces with a finite alphabet  and not only for coded shifts.
\begin{proposition}\label{lem:pres_upper}
Let $X$ be a subshift given by the oracle $\psi$. Then the topological pressure $\Ptop(X,\cdot)$ is upper semi-computable on $C(X,\bR)$. 
\end{proposition}
\begin{proof} Let $\phi\in C(X,\bR)$ be given by an oracle $\chi$, and fix $n\in \bN$.   We first show that the $n$-th partition function $Z_n(\phi)$, see the definition in Equation \eqref{defnpart}, is computable.  Let $w\in \cL(X,n)$, and consider a partition of $[w]$ (in $X$) into small cylinders by listing all words beginning with $w$ in $\cL(X,n')$ for fixed $n'>n$.  Using this subdivision, we can approximate  $\sup_{x\in [w]}S_n\phi(x)$, see Definition \eqref{eq:Sn}, to any given precision by approximating both $\phi$ and its variation within each cylinder.  Repeating this procedure for all words $w\in \cL(X,n)$, and using that the exponential function is computable, we conclude that the  $n$-partition function $Z_n(\phi)$ can be computed 
to any given precision from the oracles $\psi$ and $\chi$.  It now follows from the computability of  the logarithm function that $\frac{1}{n}\log   Z_n(\phi)$ is computable.
By Equation \eqref{eqn:def:P}, $\frac{1}{n}\log   Z_n(\phi)$ converges to $\Ptop(\phi, X)$ from above as $n$ goes to $\infty$ (see, e.g., \cite{Walters1973}).  Therefore, by taking the minimum of approximations for $\frac{1}{n}\log   Z_n(\phi)$ for a strictly increasing sequence $(n_i)_i$, we obtain a non-increasing sequence of rational numbers which converges to $\Ptop(X,\phi)$ from above.
\end{proof}

\subsection{Lower semi-computability of the topological pressure}
We prove that the topological pressure for coded shifts satisfying the conditions of Theorem~A is lower semi-computable.
This completes the proof of Theorem~A when combined with Proposition \ref{lem:pres_upper}. 

Let $f:X\to X$ be a coded shift with unique representation $\cG=\{g_1,g_2,\dots\}$. 
In order  to prove Part (i) in Theorem A, we identify potentials $\phi\in C(X,\bR)$  that satisfy  
\begin{equation}\label{limit}\Ptop(\phi)=\lim_{m\to\infty} \Ptop(X_m,\phi),
\end{equation} where $X_m=X(\{g_1,g_2,\dots, g_m\})$. 
For the zero potential, this property reduces to $X$ being an almost Sofic shift, see \cite{MarcusLind1995,Petersen1986}. We define 
\begin{equation*}
 X_{\rm fin} = \bigcup_{m\in \bN} X_m\,\,\, {\rm and}\,\,\, X_{\rm lim}=X\setminus X_{\rm seq}.
\end{equation*}
To obtain Equation \eqref{limit}, we  consider the following two quantities:
 \begin{equation*}
 \begin{split}
P_{\rm seq}(\phi)&\eqdef \sup\{P_\mu(\phi): \mu\in \cM\,\, {\rm and }\,\, \mu(X_{\rm seq})=1\},\\
P_{\rm lim}(\phi)&\eqdef \sup\{P_\mu(\phi): \mu\in \cM\,\, {\rm and }\,\, \mu(X_{\rm lim})=1\}.
\end{split}
\end{equation*} 

We recall some results from ergodic theory that are used in the proof of Theorem~A.  
 Let $f:X\to X$ be a subshift,  $\mu\in \cM_X$, and  $E$ be a Borel subset of $X$  with $\mu(E)>0$. We define the first return time function\footnote{There are situations where it is more efficient to work with general return times rather than the first return time, see, e.g., \cite{PesinSenti2008}.}  on $E$ by
$ r_E(x)=\min\{n\geq 1: f^n(x)\in E\}$, 
where $\min \emptyset=+\infty$. Let $E^\infty$ denote the set of $x\in E$ such that $f^k(x)\in E$ for infinitely many $k\in \bN$. The set $E^\infty$ is measurable, and, by Poincar\'e's Recurrence Theorem, we have $\mu(E^\infty)=\mu(E)$.  
The {\em induced map} $f_E:E^\infty\to E^\infty$ is defined by
\begin{equation}\label{deffE}
f_E(x)\eqdef f^{r_E(x)}(x).
\end{equation}
It is straight-forward to verify that the following measure is an $f_E$-invariant probability measure:
\begin{equation}\label{defindmeas}
\mu_E(F)\eqdef \mu(F|E)=\frac{\mu(F\cap E)}{\mu(E)}.
\end{equation}
When $E$ is clear from context, we write $i(\mu)=\mu_E$. We use the notation $ \cM_{f_E} $ and $\cM_{{\rm erg},f_E}$ for the  Borel $f_E$-invariant and ergodic probability measures on $E$, respectively, even though, in general, $f_E$ is not continuous. An important property due to Kac \cite{Kac1947} is that if $\mu$ is ergodic, then $i(\mu)$ is ergodic, and 
\begin{equation}\label{eqKac}
\int_X \phi\, d\mu=\int_E\left(\sum_{k=0}^{r_E(x)-1}\phi\circ f^k(x)\right) d\mu
\end{equation}
for all $\mu$-integrable functions $\phi$. Moreover,
$ \int_E r_E\, d i(\mu)=1/\mu(E).$ 
Abramov's theorem \cite{Abramov1959} relates the entropy of the return map with the entropy of the original map as 
\begin{equation}\label{eq:abramov}
h_{i(\mu)}(f_E)=\frac{h_\mu(f)}{\mu(E)}.
\end{equation}
The following proposition provides the tool to reconstruct an $f$-invariant measure  from an $f_E$-invariant measure, see, e.g., \cite[Proposition 1.1]{Zweimuller2005}. Here, we only consider  the special case of probability measures. 

\begin{proposition}\label{propAa} 
Let $E\subset X$ be measurable and let $\nu\in \cM_{f_E}$ with $\int_E r_E\, d\nu<\infty$. Then $\widetilde{\nu}$, defined by 
\begin{equation*}
\widetilde{\nu}(B)\eqdef \sum_{k=0}^\infty \nu \left(\{x: r_E(x)>k\}\cap f^{-k}(B)\right),
\end{equation*}
is a finite  $f$-invariant measure with $\widetilde{\nu}(X)= \int_E r_E\, d\nu$. Moreover, $l(\nu)\eqdef \frac{1}{\widetilde{\nu}(X)}\, \widetilde{\nu}\in \cM_f$ with $i(l(\nu))=\nu$. Furthermore,  if $\nu$  is ergodic, then $l(\nu)$ is ergodic.
\end{proposition}
When the hypotheses of Proposition \ref{propAa} hold, we call $l(\nu)$ the lift of $\nu$ and denote the set of liftable  $f_E$-invariant measures by 
\[\cM_{f_E, {\rm lift}}=\left\{\nu\in \cM_{f_E}: \int r_E\, d\nu<\infty\right\}.\]

We briefly review relevant material from the thermodynamic formalism for countable shift spaces following the monograph by Mauldin and Urbanski \cite{MauldinEtAl2003}. For related results, we refer the reader to  Sarig \cite{Sa99}.
We note that  \cite{MauldinEtAl2003} treats the case of one-sided 
countable shift spaces. Here we present analogous results in the two-sided  case. The proofs in the two-sided setting are entirely analogous. Furthermore, we only consider  the  full shift  case which is sufficient for our purposes.  

Let $Y=Y^{\pm}_\bN$ denote the shift space of bi-infinite sequences $(y_k)_{k\in \bZ}$ in the  countable alphabet $\bN$. We endow $Y$  with the analog of the $\theta$-metric (see Equation \eqref{defmetX}) which makes $Y$ into a non-compact Polish space.
Let $g:Y\to Y$ denote the (left) shift on $Y.$ We continue to use the notation from Section 2.1 for  the countable shift space. Given a continuous function $\phi:Y\to \bR$, the topological pressure $P_{\rm top}(\phi)=  P_{\rm top}(Y,\phi)$ of $\phi$ is defined as in the case of finite-alphabet shift spaces, see Equations \eqref{defnpart}, \eqref{eq:Sn} and \eqref{eqn:def:P}.  Given $m\in \bN$, we consider $Y_m=\{(y_k)_{k\in \bZ}: y_k\in \{1,\dots,m\}\}$ as a (compact) subshift of $Y$. We say a function $\phi:Y\to \bR$ is acceptable if
it is uniformly continuous and has finite oscillation, i.e., 
\[
{\rm osc}(\phi)=\sup\left\{\sup(\phi\vert_{[e]})-\inf(\phi\vert_{[e]}): e\in \bN\right\}<\infty.
\]
We note that an acceptable function does not need to be bounded.
We denote the set of all Borel $g$-invariant and ergodic probability measures on $Y$ by $\cM_g$ and $\cM_{{\rm erg},g}$, respectively. 
\begin{theorem}[see \cite{MauldinEtAl2003}]\label{countUM}
Let $\phi:Y\to\bR$ be acceptable. Then 
\begin{equation}\label{varpricount}
P_{\rm top}(\phi)=\sup
\left\{h_\mu(g)+\int \phi\, d\mu: \mu\in \cM_g, \int \phi \, d\mu>-\infty\right\}=\sup\left\{P_{\rm top}(Y_m,\phi): m\in \bN\right\}.
\end{equation}
Moreover, the statement remains true if  the first supremum is  taken only over all Borel $g$-invariant ergodic  probability measures.
\end{theorem}

As in the case of finite-alphabet shift spaces, we refer to the left-hand side  identity in Equation \eqref{varpricount} as the variational principle for the countable full shift. We say that $\phi:Y\to \bR$ is constant on centered $k$-cylinders
(and write $\phi\in LC_k(Y,\bR)$) if $\phi\vert_{[y]_{-k}^k}$ is constant for all $y\in Y$. We denote by $LC(Y,\bR)=\bigcup_k LC_k(Y,\bR)$ the set of locally constant potentials on $Y$. Evidently, each locally constant potential
is uniformly continuous.

 The following result is the main tool to establish the lower computability of the topological pressure for coded shifts. It is  also of independent interest outside the area of computability.

\begin{proposition}\label{lem:case1}
Let $f:X\to X$ be a coded shift with unique representation $\cG=\{g_1,g_2,\dots\}$. 
Let $\phi\in LC(X,\bR)$ and suppose there exists $\mu_\phi\in ES(\phi)$ with $\mu(X_{\rm seq})=1$. Then $
  P_{\rm top}(\phi)= P_{\rm seq}(\phi)=\lim_{m\to\infty} \Ptop(X_m,\phi).$
\end{proposition} 
\begin{proof}
The identity $P_{\rm top}(\phi)= P_{\rm seq}(\phi)$ is trivial since  $P_{\rm top}(\phi)= P_{\mu_\phi}(\phi)=P_{\rm seq}(\phi)$.
Since $\phi$ is locally constant, there exists $k\in \bN_0$ such that  $\phi$ is constant on cylinders $[x]_{-k}^k$  for all $x\in X$, i.e., $\phi\in LC_k(X,\bR)$.  By applying an ergodic decomposition argument, we may assume that $\mu_\phi$  is ergodic.
We note that $P_{\rm top}(\phi+c)=P_{\rm top}(\phi)+c$ and $P_{\rm top}(X_m,\phi+c)=P_{\rm top}(X_m,\phi)+c$  for all $c\in \bR$. Therefore, by replacing the potential $\phi$ with $\phi-P_{\rm top}(\phi)$, it suffices to prove
the statement for the case $P_{\rm top}(\phi)=0$. To obtain this reduction, we also use the fact that $ES(\phi)=ES(\phi-P_{\rm top}(\phi))$. For the remainder of this proof, we assume $P_{\rm top}(\phi)=0$.

  If $\mu_\phi(X_{\rm fin})=1$, then, since the sets $X_m$ and $X_{m+1}\setminus X_m$ are $f$-invariant and since $\mu_\phi$ is ergodic, there exists $m\in \bN$ with $\mu(X_{m})=1$, in which case the assertion holds.

Since  $X_{\rm seq}\setminus X_{\rm fin}$ is $f$-invariant it remains to consider the case $\mu_\phi(X_{\rm seq}\setminus X_{\rm fin})=1$. We define 
 \begin{equation*}
E\eqdef \{\cdots g_{i_{-2}}g_{i_{-1}}.g_{i_0}g_{i_1}g_{i_2}\cdots: g_{i_j}\in \cG\}.
\end{equation*}
In other words, $E$ is the set of all infinite sequences in the alphabet of generators $\cG$ for which a generator begins at index $0$.
Alternatively, $E$ is the set of all $x=\cdots x_{-2}x_{-1}x_0 x_1x_2\cdots \in X_{\rm seq}$ for which $x_{-1}$ is the  terminal element of the generator $g_{i_{-1}}=g_{i_{-1}}(x)$ and $x_0$ is the initial element of the generator $g_{i_0}=g_{i_0}(x)$.  
Since $X$ is uniquely representable by \cite[Corollary 8]{BPR21}, $r_E(x)=|g_{i_0}(x)|$ is well-defined.  We observe that every $x\in E$ has finite return time.  Let $f_E:E\to E$ be defined as in Equation \eqref{deffE}. Since $\bigcup_{n\geq 0}f^{-n}(E)=X_{\rm seq}$, it follows that $\mu(E)>0$ for all $\mu\in \cM$ with $\mu(X_{\rm seq})>0$.  
In particular, $\mu_\phi(E)>0$.
We define 
\begin{equation*}
\phi_E : E\to \bR \quad\text{where}\quad \phi_E(x) = \sum_{j=0}^{r_E(x)-1}\phi\circ f^j(x).
\end{equation*}

Let $\mu\in \cM$ with $\mu(X_{\rm seq})>0$.
Definition \eqref{defindmeas} and Kac's formula (Equation~\eqref{eqKac}) imply that  $\int_E \phi_{E}\, d\, i(\mu)=\frac{\int \phi\, d\mu}{\mu(E)}$.  This, when combined with Abramov's Theorem (Equation~\eqref{eq:abramov}), shows that
\begin{equation}\label{eq24}
P_{i(\mu)}(\phi_E)=\frac{P_{\mu}(\phi)}{\mu(E)}.
\end{equation}

Next, we construct a countable shift map with alphabet $\cG$ that is conjugate to $f_E$. Let $Y=\cG^{\bZ}$ and let $g:Y\to Y$ denote the shift map.  We define $h: E\to Y$ by replacing the word of each generator by its symbol in the alphabet of generators $\cG$, i.e.,
 \begin{equation*}
 h(x)=\dots g_{-2}(x)g_{-1}(x).g_0(x) g_1(x) g_2(x)\cdots .
 \end{equation*}
 Since $\cG$ is a unique representation of $X$, the following diagram commutes:
 \begin{center}
 \begin{tikzpicture}
 \node (A) at (0,0) {$Y$};
\node (B) at (2,0) {$Y$};
\node (C) at (0,2) {$E$};
\node (D) at (2,2) {$E$};
\draw[->] (A) edge node[above] {$g$} (B);
\draw[->] (C) edge node[left] {$h$} (A);
\draw[->] (D) edge node[right] {$h$} (B);
\draw[->] (C) edge node[above] {$f_E$} (D);
 \end{tikzpicture}.
 \end{center}
 For each $\eta\in \cM_{f_E}$, there exists a unique measure-theoretically isomorphic measure $\widetilde{\eta}\in \cM_g$ defined by $\widetilde{\eta}(A)=\eta(h^{-1}(A))$.   Hence
  \begin{equation}\label{eqee6}
 h_{\widetilde{\eta}}(g)=h_\eta(f_E).
 \end{equation}
 We write $\widetilde{\mu}\eqdef \widetilde{i(\mu)}\in \cM_g$ and define $\widetilde{\phi}\eqdef \phi_E \circ h^{-1}:Y\to \bR$. Thus,
 \begin{equation}\label{eq29}
 \int \widetilde{\phi} \, d\, \widetilde{\mu} = \int \phi_E\, d i(\mu).
 \end{equation} 
 Combining Equations \eqref{eq24}, \eqref{eqee6}, and \eqref{eq29} yields
 \begin{equation}\label{eq19}
 P_{\widetilde{\mu}}\left(\widetilde{\phi}\right)=\frac{P_{\mu}(\phi)}{\mu(E)}.
 \end{equation}
 
 Next we show that $\widetilde{\phi}$ is acceptable. We observe that $\widetilde{\phi}\in LC_k(Y,\bR)$. This follows from $\phi\in LC_k(X,\bR)$, the definitions of $\phi_E$ and $\widetilde{\phi}$, and  the fact that each generator in $\cG$ has length at least one.
 Therefore, $\widetilde{\phi}$ is uniformly continuous. Next we estimate the oscillation of $\widetilde{\phi}$. Define $c=\sup \phi -\inf \phi <\infty$, and let $e\in \cG$. If $|e|\leq 2k$, then, by the definition of $\phi_E$, $\sup\left(\widetilde{\phi}\vert_{[e]}\right)-\inf\left(\widetilde{\phi}\vert_{[e]}\right)\leq 2k c$. Next we consider the case $|e|> 2k$. Let $y=\cdots g_{-2}g_{-1}.e g_1 g_2\cdots \in [e]$ and $x=h^{-1}(y)=\cdots x_{-2}x_{-1}.x_0 x_1x_2\cdots \in X_{\rm seq}$. In particular, $e=x_0 \dots x_{|e|-1}$.
 Then
 \begin{equation}\label{eqsumacc}
 \widetilde{\phi}(y)= \sum_{j=0}^{k-1}\phi\circ f^j(x) +  \sum_{j=k}^{|e|-k-1}\phi\circ f^j(x) +  \sum_{j=|e|-k}^{|e|-1}\phi\circ f^j(x).
 \end{equation}
 We observe that $\phi\in LC_k(X,\bR)$ implies that $\sum_{j=k}^{|e|-k-1}\phi\circ f^j(x)$ is  independent of $x\in [e]$. Thus Equation \eqref{eqsumacc} implies that
  $\sup\left(\widetilde{\phi}\vert_{[e]}\right)-\inf\left(\widetilde{\phi}\vert_{[e]}\right)\leq 2k c$. Hence the oscillation is finite and  $\widetilde{\phi}$ is acceptable.
  Thus, Theorem \ref{countUM} yields
  \begin{equation}\label{eqgut}
  P_{\rm top}\left(\widetilde{\phi}\right)=\sup\left\{P_{\rm}\left(Y_m,\widetilde{\phi}\right): m\in \bN\right\}=\lim_{m\to\infty}P_{\rm}\left(Y_m,\widetilde{\phi}\right).
  \end{equation}
  Next we prove that $P_{\rm top}\left(\widetilde{\phi}\right)=0$. Since $P_{\rm top}(\phi)=\P_{\mu_\phi}(\phi)=0$, Equation \eqref{eq19} applied to $\mu_\phi$ shows that $P_{\widetilde{\mu}_\phi}\left(\widetilde{\phi}\right)=0$.  
  Thus, $P_{\rm top}\left(\widetilde{\phi}\right)\geq 0$ follows from the variational principle for the countable full shift Equation \eqref{varpricount}. To establish $P_{\rm top}\left(\widetilde{\phi}\right)\leq 0$,
  it suffices to show that $P_{\rm top}\left(Y_m,\widetilde{\phi}\right)\leq 0$ for all $m\in\bN$, by   Equation \eqref{eqgut}.  For fixed $m\in \bN$, let $\widetilde{\eta}_m=\widetilde{\eta}_{m,\widetilde{\phi}}$ be an ergodic equilibrium measure of $\widetilde{\phi}\vert_{Y_m}$. In fact, $\widetilde{\eta}_m$ is the unique
  equilibrium state of $\widetilde{\phi}\vert_{Y_m}$ since $Y_m$ is a transitive Sofic shift and  $\widetilde{\phi}\vert_{Y_m}$ is H\"older continuous, see \cite{Yoo2018}. Let $\eta_m\in \cM_{f_E}$ be defined by $\eta_m(A)=\widetilde{\eta}_m(h(A))$. 
  In particular, $\eta_m$ and $\widetilde{\eta}_m$ are measure-theoretically isomorphic.
  Hence $P_{\eta_m}(\phi_E\vert_{X_m\cap E})= P_{\rm top}\left(Y_m,\widetilde{\phi}\right).$ Since $\cG_m$ is finite, $\int r_E \,d\eta_m<\infty$ and thus $\eta_m\in \cM_{f_E, {\rm lift}}$. We define $\mu_m=l(\eta_m)$, that is, $\mu_m$ is the lift of $\eta_m$ to $X_m$. Since
 \begin{equation}\label{eq2222}
 P_{\mu_m}(\phi)\leq P_{\rm top}(X_m,\phi)\leq P_{\rm top}(\phi)=0,
 \end{equation}
 and
 \begin{equation}\label{gqrett}
 P_{\eta_m}(\phi_E)=\frac{P_{\mu_m}(\phi)}{\mu_m(E)}
 \end{equation}
 (see Equation \eqref{eq24}), we conclude that $P_{\eta_m}(\phi_E)=P_{\widetilde{\eta}_m}\left(\widetilde{\phi}\right)\leq 0$. This completes the proof of $P_{\rm top}\left(\widetilde{\phi}\right)= 0$. Combining this identity with 
 Equation \eqref{eqgut}  yields
 \begin{equation}\label{eqgutter}
 \lim_{m\to\infty} P_{\widetilde{\eta}_m}\left(\widetilde{\phi}\right)=\lim_{m\to\infty}P_{\rm}\left(Y_m,\widetilde{\phi}\right)=0
 \end{equation}
 Finally, by combining Equations \eqref{eq2222}, \eqref{gqrett}, and \eqref{eqgutter} with $P_{\eta_m}(\phi_E\vert_{X_m\cap E})= P_{\rm top}\left(Y_m,\widetilde{\phi}\right)$, we conclude that $\lim_{m\to\infty} P_{\rm top}(X_m,\phi)=0=P_{\rm top}(\phi).$
 \end{proof}
 
\begin{theorem}\label{thmalmsofpres} Let $f:X\to X$ be a coded shift with unique representation $\cG=\{g_1,g_2,\dots\}$.  Suppose that $\phi\in FSSP(X,\bR)$, then
  $
  P_{\rm top}(\phi)=\lim_{m\to\infty} \Ptop(X_m,\phi).$
  \end{theorem}
 \begin{proof} 
 Let $\phi\in FSSP(X,\bR)$. We observe that for fixed $\eta\in \cM$,  the map $\varphi\mapsto P_\eta(\varphi)$ is $1$-Lipschitz continuous on $C(X,\bR)$ with respect to the  supremum norm. This implies that the 
 map $\varphi \mapsto P_{\rm top}(\varphi)$ is also $1$-Lipschitz continuous on $C(X,\bR)$ with respect to the  supremum norm. Fix $\epsilon>0$. By the definition of $FSSP(X,\bR)$ and since $LC(X,\bR)$ is dense
 in  $C(X,\bR)$, there exists $\varphi\in LC(X,\bR)$  with $\Vert \phi-\varphi\Vert_{\infty}<\frac{\epsilon}{3}$ and $P_{\rm seq}(\varphi)>P_{\rm lim}(\varphi)$. This shows that any equilibrium state $\mu$ of $\varphi$ satisfies $\mu(X_{\rm seq})=1$. Since
  $\varphi$ has at least one equilibrium state, by applying Proposition \ref{lem:case1}  to $\varphi$,  there exists $m\in \bN$  such that $\vert P_{\rm top}(X_m,\varphi)-P_{\rm top}(\varphi)\vert <\frac{\epsilon}{3}$.  Putting these statements together we conclude that
 \begin{align*}
 \vert P_{\rm top}(\phi)&-P_{\rm top}(X_m,\phi)\vert\\
 &< \vert P_{\rm top}(\phi)-P_{\rm top}(\varphi)\vert 
  + \vert P_{\rm top}(\varphi)-P_{\rm top}(X_m,\varphi)\vert+\vert  P_{\rm top}(X_m,\varphi) -P_{\rm top}(X_m,\phi)\vert <\epsilon
 \end{align*}
 which completes the proof.
 \end{proof}
 
 We now use these results to prove  Theorem A.

\begin{proof}[Proof of Theorem A] Part (i) of Theorem A is proven in Theorem \ref{thmalmsofpres}.
  To prove the computability statement for the topological pressure on  $FSSP(X)$,  by Proposition \ref{lem:pres_upper}, it suffices to prove that $\phi\mapsto \Ptop(\phi)$ is lower semi-computable on  $FSSP(X)$.  We observe that the lower-computability of $\Ptop(\cdot)$ on  $FSP(X)$ follows from Part (i) of Theorem A and the fact that $\Ptop(X_m,\phi)$ is computable on $C(X,\bR)$ (see Theorem \ref{cSofic}). 
  
To complete the proof, we  address the non-computability statement on the complement of $FSP(X)$. Let $\phi\in C(\Sigma,\bR)$ such that $\phi|_X\in C(X,\bR)\setminus FSP(X)$ and let $\epsilon\eqdef \Ptop(\phi)-P_{\rm seq}(\phi)= P_{\rm lim}(\phi)-P_{\rm seq}(\phi)>0$.    
We consider the function $Y\mapsto\Ptop(Y,\phi)$ for $Y$ in a neighborhood of $X$ in $\Sigma_{\rm coded}$.  Suppose that there is an oracle Turing machine $\psi$ for this function which takes, as input, a precision as well as oracles $\chi_{\cL(X)}$ and $\chi_{\cG}$ for the language $\cL(X)$ and generating set $\cG$, respectively.  Let $n$ and $m$ be the precisions to which $\psi$ queries the oracles $\chi_{\cL(X)}$ and $\chi_{\cG}$, respectively, in the computation of $\psi\left(\frac{\varepsilon}{2},\chi_{\cL(X)},\chi_{\cG}\right)$.  By choosing a larger $m$, if necessary, we may assume, without loss of generality, that $X_m$ is in the neighborhood of $X$ specified above and that $\cL(X_m,n)=\cL(X,n)$.  In this case, there are oracles $\chi_{\cL(X_m)}$ and $\chi_{\cG_m}$ which agree with $\chi_{\cL(X)}$ and $\chi_{\cG}$ up to precision $m$ and $n$, respectively.  Then, $\psi\left(\frac{\varepsilon}{2},\chi_{\cL(X)},\chi_{\cG}\right)=\psi\left(\frac{\varepsilon}{2},\chi_{\cL(X_m)},\chi_{\cG_m}\right)$ since the queries up to precision $n$ and $m$ are identical.

Suppose that $\mu_{X_m}\in\cM_{X_m}$.  We denote the (formal) extension of $\mu_{X_m}$ to $X$  by $\mu$  and observe that $\mu\in\cM_{X}$.  In addition, 
$$
h_{\mu_{X_m}}(f)+\int \phi|_{X_m} d\mu_{X_m}=h_{\mu}(f)+\int \phi|_X d\mu.
$$
Therefore, it follows that $\Ptop(X_m,\phi)\leq P_{\rm seq}(\phi)\leq \Ptop(X,\phi)-\epsilon$.  This leads to a contradiction, however, since the existence of $\psi$ implies that
\begin{align*}
\left|\Ptop(X,\phi)\right.&\left.-\Ptop(X_m,\phi)\right|\\
&=\left|\Ptop(X,\phi)-\psi\left(\frac{\varepsilon}{2},\chi_{\cL(X)},\chi_{\cG}\right)+\psi\left(\frac{\varepsilon}{2},\chi_{\cL(X_m)},\chi_{\cG_m}\right)-\Ptop(X_m,\phi)\right|<\epsilon.\\[-1.26cm]
\end{align*}
\end{proof}

\section{Computability of the topological pressure for general shifts} \label{sec:proof:counterexample}

We prove Theorem B and Corollary \ref{cor4}. We begin with two preparatory results. 

\begin{lemma}\label{lem:posentropy}
Let $X\in\Sigma_{\rm invariant}$ and let $\phi\in C(X,\bR)$.  Define $b_\phi=\sup\{\int\!\phi\, d\mu: \mu\in \mathcal{M}_X\}$.  Then, $\Ptop(\phi)>b_\phi$ if and only if  for all $\mu\in ES(\phi)$, $h_\mu(f)>0$.
\end{lemma}
\begin{proof}
Let $\mu\in ES(\phi)$, i.e., $\Ptop(\phi)-\int\!\phi\, d\mu=h_\mu(f)$.  If $\Ptop(\phi)>b_\phi$, then $h_\mu(f)\geq\Ptop(\phi)-b_\phi>0$.  On the other hand, suppose $\Ptop(\phi)=b_\phi$. By the compactness of $\mathcal{M}$,  the supremum in the definition of $b_\phi$ is a maximum. Thus, there exists $\mu\in\mathcal{M}$ with $\int\!\phi\, d\mu=b_\phi$, and, thus, $\mu\in ES(\phi)$. We conclude  that $h_\mu(f)=0$.
\end{proof}

Let $X\in\Sigma_{\rm invariant}$ and consider a potential $\phi\in C(\Sigma,\bR)$ such that $f:X\to X$ and $\phi|_X$ satisfy the conditions in Lemma~\ref{lem:posentropy}.  Let $\varepsilon>0$ such that $\Ptop(\phi|_X)-b_{\phi|_X}>5\varepsilon$ and fix $n\in \bN$. Let $\calF_n$ be the complement of $\calL(X,n)$ in $\cA_d^n$ and let $X_{\mathcal{F}_n}$ be the SFT whose forbidden words are given by $\calF_n$.  We recall that a transitive component of $X_n$ is a maximal transitive invariant subset of $X_{\mathcal{F}_n}$.

\begin{lemma}\label{lem:smallintegral}
Suppose that $Y$ is a  transitive component of $X_n$ with $X\cap Y\not=\emptyset$. For any $w\in \calL(Y)$, there exists a periodic point $y\in Y$ which contains  the word $w$ and satisfies $\int\!\phi\,d\mu_y< b_{\phi|_X}+4\varepsilon$.
\end{lemma}
\begin{proof}
We observe that $X\cap Y$ is a subshift. Therefore, by the Krylov-Bogolyubov theorem, see, e.g., \cite{Walters1973}, $ \mathcal{M}_{Y\cap X}\not=\emptyset$. Moreover, an ergodic decomposition argument shows that there exists
 $\mu\in\calM_{{\rm erg},Y\cap X}$. Let $z\in Y$ be a Birkhoff generic point of $\mu$.  Then, there exists $M\in \bN$ such that for all $m\geq M$,
$$\left|\frac{1}{m}S_m\phi(z)-\int\!\phi\,d\mu\right|<\varepsilon,$$
see Definition \eqref{eq:Sn}.  Since $\phi$ is continuous and $\Sigma$ is compact, $\phi$ is uniformly continuous.  
Therefore, for all $m\geq M$, there exists a $k\in\bN$ such that for $y\in [z]_{-k}^{m+k}$, we have
\begin{equation}\label{eq114}
\left| S_{m+1}\phi(y)- S_{m+1}\phi(z)\right|<(m+1)\varepsilon.
\end{equation}
If follows from the transitivity of $Y$ that for any word $w'$ in the language for $Y$, there exists words $v_p=v_p(w')$ and $v_s=v_s(w')$ in the language of $Y$ such that $\calO(w'v_pwv_s)\in Y$.    
Moreover, since $Y$ is a transitive SFT, the lengths of $v_p$ and $v_s$ are uniformly bounded from above independently of $w'$.  Let $C$ be  a uniform upper bound for $|v_p|+|w|+|v_s|+2k$.
Fix $m\geq M$ such that
$$\frac{C}{C+m+1}\|\phi\|_\infty<\varepsilon.$$
We note that this condition implies that $m$ is  large with respect to $C$.  Fix $w'=z_{-k} \cdots z_{m+k}$, where $k$ is given as in Inequality \eqref{eq114}, and let  $v_p=v_p(w')$ and $v_s=v_s(w')$ be as above.  We write $y=f^{-k}(\calO(w'v_pwv_s))$ and observe that $y\in[z]_{-k}^{m+k}$.  Moreover, let $c=|v_p|+|w|+|v_s|+2k$ and $\ell=c+m+1$, which is the period of $y$.  Then, since $\int\!\phi\, d\mu_y=\frac{1}{\ell}S_{\ell}\phi(y)$,
\begin{align*}
\left|\int\!\phi\,d\mu_y-\int\!\phi\,d\mu\right|&\leq\frac{1}{\ell}\left|S_\ell\phi(y)-S_\ell\phi(z)\right|+\left|\frac{1}{\ell}S_\ell\phi(z)-\int\!\phi\,d\mu\right|\\
&<\frac{1}{\ell}\left|S_{m+1}(y)-S_{m+1}\phi(z)\right|+\frac{1}{\ell}\left|S_c\phi(f^{m+1}(y))-S_c\phi(f^{m+1}(z))\right|+\varepsilon\\
&<\frac{(m+1)}{\ell}\varepsilon+\frac{2c}{\ell}\|\phi\|_\infty+\varepsilon<4\varepsilon.
\end{align*}
The result now follows since $\int\!\phi\,d\mu$ is bounded above by $b_{\phi|_X}$.
\end{proof}

Now we prove Theorem B. 
\begin{proof}[Proof of Theorem B]
We denote the transitive components of $X_n$ which have  nonempty intersection with $X$ by $Y_{n,1},\dots,Y_{n,c_n}$, see \cite[Chapter 5]{Kit} for details. Furthermore, for each component  $Y_{n,i}$, we fix a periodic point $y_{n,i}\in X_n$ whose language is $\calL(Y_{n,i},n)$, as constructed in Lemma \ref{lem:smallintegral}.  This can be done by applying Lemma  \ref{lem:smallintegral} to  a word $w=w_{n,i}$ that contains all  words in $\calL(Y_{n,i},n)$, and such a word exists since $Y_{n,i}$ is transitive.  We denote the orbit of $y_{n,i}$ by $O_{n,i}$, i.e., the finite collection of shifts of $y_{n,i}$, and let $Z_n'=\bigcup_i O_{n,i}$. We construct a subshift $Z_n$ with the property that $\calL(Z_n,n)=\calL(X_n,n)=\calL(X,n)$, but $\Ptop(Z_n,\phi)\leq b_{\phi|_X}+4\varepsilon<\Ptop(X,\phi)-\varepsilon$ for fixed $\varepsilon$.  This condition shows that the pressure is not continuous at $X$, and, hence, is not computable.  The shift $Z_n$ is constructed in such a way that $Z_n'$ is the nonwandering set of $Z_n$.

For each $w\in\calL(X,n)$, we fix $x_w\in X$ such that $w$ is a word in $x_w$.  If $x_w$ is a nonwandering point then $x_w\in Y_{n,i}$ for some $i=1,\dots,c_n$, see \cite{Kit}. Hence $w\in \calL(O_{n,i},n)\subset \calL(Z_n',n)$.

We next consider the case when $x_w$ is a wandering point of $X$. Let $Y_{n,w}^{\omega}$ and $Y_{n,w}^\alpha$ be the transitive components of $X_n$ that contain the omega and alpha limit sets of $x_w$, respectively, see \cite{Kit} for additional details.  Since $x_w$ is wandering,  $Y_{n,w}^{\omega}\not=Y_{n,w}^\alpha$.  We observe that since $x_w$ has accumulation points in both $Y_{n,w}^{\omega}$ and $Y_{n,w}^\alpha$, $Y_{n,w}^\omega\cap X$ and $Y_{n,w}^\alpha\cap X$ are nonempty.  In other words, $Y_{n,w}^\omega=Y_{n,i}$ and $Y_{n,w}^\alpha=Y_{n,j}$ for some $i,j$ with $i\not= j$.  Let $\tau_{n,w}^\omega$ and $\tau_{n,w}^\alpha$ be the generating segments of the periodic points $y_{n,i}$ and $y_{n,j}$, respectively.  We now construct a wandering point $y_{n,w}=(\tau_{n,w}^\alpha)^\infty v_pwv_s(\tau_{n,w}^\omega)^\infty\in X_n$ for some $v_p$ and $v_s$, which is eventually
periodic under both forward and backward iteration.  The challenge in this construction is to prove the existence of $v_p$ and $v_s$ so that $y_{n,w}\in X_n$, i.e., so that no forbidden words appear in this point.  For the forward direction, we observe that for $i$ sufficiently large, $w$ appears in $x_w$ with index less than $i$ and $(x_w)_i\dots(x_w)_{i+n-1}$ is a word in $\calL(Y_{n,w}^\omega,n)$.  Since $Y_{n,w}^\omega$ is transitive, there is some word in $Y_{n,w}^\omega$ which connects $(x_w)_i\dots(x_w)_{i+n-1}$ to $\tau_{n,w}^\omega$.  Therefore, we conclude that $w$ can be connected to $\tau_{n,w}^\omega$ without using forbidden words.  For the backwards direction we proceed similarly.

Finally, we define $Z_n=Z_n'\cup \bigcup_{x_w\, {\rm wandering}} O(y_{n,w})$, which is a subshift
with $\calL(Z_n,n)=\calL(X_n,n)$. Moreover, the
nonwandering set of $Z_n$ is $Z_n'$. Using the fact that the pressure coincides with the pressure restricted to the nonwandering set, see, e.g., \cite{Walters1973}, we conclude that
\[
P_{\rm top}(Z_n,\phi)=P_{\rm top}(Z_n',\phi)=\max_i P_{\rm top}(O_{n,i},\phi)<b_{\phi\vert_{X_n}}+4\varepsilon< P_{\rm top}(X,\phi)-\varepsilon.\vspace{-0.73cm}
\]
\end{proof}

Finally, we present the proof of Corollary \ref{cor4}.

\begin{proof}[Proof of Corollary \ref{cor4}]
Let $\phi_0=0$ and $X_0\in \Sigma_{\rm invariant}$. If $h_{\rm top}(X_0)=P_{\rm top}(X_0,\phi_0)>0$ then, by the variational principle in Definition \eqref{varpri}, every equilibrium state of $\phi_0$ is a measure of maximal entropy, and, thus, has positive entropy. Therefore,
Theorem B  implies that $X\mapsto h_{\rm top}(X)$ is not computable at $X_0$. 

Assume now that $h_{\rm top}(X_0)=0$. Let $\chi_{\cL(X_0)}$ be an oracle for $\cL(X_0)$. Since $x\mapsto \log x$ is computable there exists a Turing machine $\psi$ which takes as input the oracle $\chi_{\cL(X_0)} $ and $m\in \bN$ and outputs a positive rational number $q_m$ with $2^{-m}>q_m-h_m\geq 0$, where 
\[
h_m=\min\left\{\frac{1}{k}\log |\cL_k(X_0)|: k=1,\dots,m\right\}.
\] 
Moreover, by computing the logarithm with increasing precision, we can assure that the sequence $(q_m)_m$ is non-increasing.
It follows from Definition \eqref{eqn:def:P} that $h_m$ is non-increasing and converges to $h_{\rm top}(X_0)=0$. We conclude that the Turing machine $\psi$  produces a non-increasing sequence of rational numbers $q_m$ that converges from above to $0$. 
Moreover,  since the entropy is nonnegative, we may compute an $m$ so that $|q_m-h_{\rm top}(X_0)|=q_m<2^{-n}$. Let $\ell_n$ be the largest precision to which the oracle   $\chi_{\cL(X_0)}$ is queried by $\psi$ so that $q_m<2^{-n}$. Let $X\in \Sigma_{\rm invariant}$ be a shift space such that there exists an oracle of $\cL(X)$ that agrees with the oracle of $\cL(X_0)$ up to precision $\ell_n$. It follows from Definition \eqref{eqn:def:P} that $q_m$ is also an upper bound of $h_{\rm top}(X)$. Hence $|q_m-h_{\rm top}(X)|<2^{-n}$. This shows that the function $X\mapsto h_{\rm top}(X)$ is computable at $X_0$. Since the Turing machine $\psi$ approximates $h_m$ to any desired precision as $m$ grows, it follows that $\psi$ uniformly computes the topological entropy for all shift spaces with zero entropy. 
\end{proof}

\section{Computability of the topological pressure of particular subshifts} \label{sec:applications}

We  apply Theorem A  to obtain  computability results for the topological pressure for particular classes of coded shifts. To derive these results, we establish the identity $FSSP(X)=C(X,\bR)$ which, together with Theorem A, imply the computability of the pressure for all continuous potentials.  We construct explicit examples of coded shifts for which $FSSP(X)\not=C(X,\bR)$ in Section 8.
In the following elementary fact we use  the standard notation $d(A,B)=\inf\{d(x,y), x\in A, y\in B\}$ for sets  $A,B\subset X$.

\begin{proposition}\label{proptriv} Let $X$ be a coded shift with generating set $\Gen$. Suppose that for all $\mu\in\cM_X$ with $\mu(X_{{\rm lim}})=1$,  $h_\mu(f)=0$. Furthermore, assume that for any $\phi\in C(X,\bR)$  with 
$P_{\rm top}(\phi)=P_{\rm lim}(\phi)$,
there exists a sequence of periodic point measures $\mu_k$ supported on $X_{\rm seq}$ such that 
\begin{equation}\label{eqdist}d\left({\rm supp}\mu_k, \cup_{\mu(X_{\rm lim})=1} {\rm supp}\mu\right)>0\end{equation}
for all $k\geq 1$ and
$P_{\mu_k}(\phi)\rightarrow P_{\rm top}(\phi)$.   Then $FSSP(X)=C(X,R)$. 
\end{proposition}

\begin{proof}
If $\phi\in C(X,\bR)$  with 
$P_{\rm top}(\phi)>P_{\rm lim}(\phi)$ then it follows from the definition of $FSSP(X)$ that $\phi\in FSSP(X)$. 
 Next we consider the case
$P_{\rm top}(\phi)=P_{\rm lim}(\phi)$. 
 By the variational principle, see Equation \eqref{varpri}, $$P_{\mu_k}(\phi)=\int \phi\, d\mu_k\rightarrow P_{\rm top}(\phi)=P_{\rm lim}(\phi)=\sup_{\mu(X_{\rm lim})=1}\int \phi\, d\mu.$$ 
 Let $\epsilon>0$ and let $k=k(\epsilon)\in \bN$ such that $|P_{\mu_k}(\phi)-P_{\rm top}(\phi)|<\epsilon$. 
  By  Urysohn's lemma and  Equation \eqref{eqdist}, there exists $\psi=\psi(\epsilon)\in C(X,\bR)$ such that 
 
\begin{enumerate}
\item $\Vert \phi-\psi \Vert_{\rm sup}=\epsilon$,
\item $\psi(x)=\phi(x)+\epsilon$ for all $x\in {\rm supp}\, \mu_k$,  and
\item $\psi(x)=\phi(x)$ for all  $x\in \bigcup_{\mu(X_{\rm lim})=1} {\rm supp}\,\mu$.
\end{enumerate}
By the second property, we have 
$$P_{\rm seq}(\psi)\geq P_{\mu_k}(\psi)= P_{\rm \mu_k}(\phi)+\epsilon>P_{\rm top}(\phi)=P_{\rm lim}(\phi). $$
Moreover, by the third property we have
$P_{\rm lim}(\psi)=P_{\rm lim}(\phi).$ 
Hence $\psi\in FSSP(X)$. Finally, by the first property and letting $\epsilon$ go to zero, we conclude  $\phi\in FSSP(X)$. 
\end{proof}

\subsection{S-gap shifts and generalizations}
S-gap shifts are symbolic systems with many practical applications including in the coding of data. They have recently received attention in dynamical systems  since they  provide simple examples of symbolic systems with some surprising properties, see, e.g., Lind and Marcus \cite{MarcusLind1995}.
\begin{definition}[S-gap shifts] 
The S-gap shift $X_S$ associated to a set $S\subset\bN_0$ is the coded shift with generating set  $\Gen = \left\{ 0^s1 \;:\; s\in S \right\}$.  
\end{definition}
We note that $S$ may contain $0$, in which case $1\in\Gen$. We observe that when $S$ is finite, the S-gap shift is a renewal shift, and, when $S$ is cofinite, $X_S$ is an SFT whose forbidden words are $\{10^s1\;:\;s\in S^c\}$.  Thus, in these cases, the computability of the topological pressure follows from Proposition \ref{cpSofic}.   

We generalize $S$-gap shifts to larger alphabets.  
For $d\geq 1$ we denote the set of permutations of $\{0,\dots,d-1\}$ by $\mathfrak{S}_d$.
We extend the class of $S$-gap shifts   by allowing permutations of the blocks of the generators.
\begin{definition}[Generalized gap shift]\label{ggs}
Let $d\geq 1$ and let $S_0,\dots, S_{d-1}\subset\bN_0$ be non-empty sets. Furthermore, let $\Pi\subset\mathfrak{S}_d$ be a non-empty subset. The  generalized gap shift $X=X_{S_0,\dots,S_{d-1},\Pi}$  associated to the sets $S_0,\dots, S_{d-1}$ and permutation set $\Pi$ is the coded shift with generating set 
\begin{equation}\label{fggs}
\Gen = \left\{\sigma(0)^{s_{\sigma(0)}}\cdots\sigma(d-1)^{s_{\sigma(d-1)}}d\;:\;  \sigma\in \Pi, s_{\sigma(j)}\in S_{\sigma(j)}\right\}. 
\end{equation}
When $|\Pi|=1$, then the order of $\{0,\dots,d-1\}$ in each generator is the same for all generators and we call the corresponding gap shift an {\em ordered gap shift}.
\end{definition}
Note that we do not require that $\Pi$ is a subgroup of $\mathfrak{S}_d$ in Equation \eqref{fggs}.
As in the case of an ordered gap shift, if the all the sets $S_0,\dots, S_{d-1}$ are finite, then  $X_{S_0,\dots,S_{d-1},\Pi}$ is a Sofic shift, in which case the topological pressure is computable.  In addition, we observe that this definition is different from $S$-limited shifts from \cite{S-limited:2018} since in generalized gap shifts, the order of $\{0,\dots,d-1\}$ may vary and the generators include the spacing character $d$. 

Spandl \cite{Spandl2007} observed that the topological entropy for S-gap shifts is computable once the set $S$ is given as input. His observation is based on a well-known formula for the topological entropy for S-gap shifts which readily  implies the computability of the entropy.  We note that our approach differs  from the one in \cite{Spandl2007},  and, in particular, it establishes the computability of the topological pressure and not merely the entropy. Moreover, it is not clear if the approach in \cite{Spandl2007} for the entropy can be applied to generalized gap shifts.

We now completely describe $X_{\lim}$ in the following result:
\begin{lemma} \label{lem:Xseq_Sgap}
Let $d\geq 1$, $S_0,\dots, S_{d-1}\subset\bN_0$, $\Pi\subset  \mathfrak{S}_d$ and let $\Gen$ be as in Definition \ref{ggs}.  Let $X=X_{S_0,\dots,S_{d-1},\Pi}$.  Suppose that ${\rm I}={\rm I}(S_0,\dots, S_{d-1})\eqdef\{i: {\rm card}(S_i)=\infty\}\not=\emptyset.$ Then  the set $X_{\lim}$ is given by the following pair-wise disjoint union:
\begin{equation*}
\begin{split} 
X_{\lim}&= \left\{ i^\infty g_{\ell_1}\cdots g_{\ell_k} j^\infty \;:\, \, g_{\ell_k}\in \Gen, i,j\in {\rm I}\right\} \\ 
& \quad  \sqcup \left\{  i^\infty g_{\ell_1}g_{\ell_2} g_{\ell_3}  \cdots \;:\;  g_{\ell_k}\in \Gen\right\}\\
 &\quad \sqcup \left\{ \cdots  g_{\ell_{3}} g_{\ell_{2}}g_{\ell_{1}}j^\infty   \;:\;  g_{\ell_k}\in \Gen, \in {\rm I}\right\}\sqcup \left\{i^\infty:\, i\in {\rm I}\right\}.
\end{split}\end{equation*}
\end{lemma}

\begin{proof}
The disjointness of the sets follows from the fact that $d\not\in I$, but appears in every $g_\ell$.
We first show that $X_{\lim}$ is contained in the disjoint union. We note that $X_{\rm seq}$ contains at most one fixed point, $d^\infty$, which is in $X_{\rm seq}$ if and only if $0\in S_i$ for all $i=0,\dots,d-1$.  All other fixed points are contained in $X_{\lim}$ and appear in the union.  Suppose now that $x\in X$ is not a fixed point. Then
\begin{equation*}
n_{\pm}(x) \eqdef \sup \left\{  n\in \bN \;:\; x_{\pm n} = d \right\}   
\end{equation*}
is well-defined.
Note that the points in $X_{\rm seq}$ are precisely those $x\in X$ for which both $n_{\pm}(x)$ are infinite. Let $x\in X\setminus X_{seq}$. Depending on the cases $(n_-(x)< \infty, n_+(x) < \infty)$, $(n_-(x) < \infty, n_+(x) = \infty)$, $(n_-(x) = \infty, n_+(x) < \infty)$, the sequence $x$  lies in the first, second or third set of the union, respectively.  The other containment is straight-forward.
\end{proof}

\begin{theorem} 
Let $f:X\to X$ be the generalized gap shift $X=X_{S_0,\dots,S_{d-1},\Pi}$.  Then $X$ is a coded shift with  unique representation $\mathcal{G}$ as in Equation \eqref{fggs}. Furthermore,  $FSSP(X)=C(X,\bR)$, and the topological pressure  on $X$ is computable when  the sets $S_0, \dots, S_{k-1}$ and the permutation set $\Pi$ are given as input.
\end{theorem}
\begin{proof}
It follows from the definition that $X$ is a coded shift with generating set $\cG$. Moreover, since for every $x=(x_k)_{k}\in X_{\rm seq}$,  we have that $x_k=d$ if and only if $x_k$ is the final letter of a generator in $\cG$, it follows that $\mathcal{G}$  is a unique representation\footnote{We note that since the sets $S_i$ may contain $0$ it is possible that a single generator in $\cG$ has multiple representations of the form of Equation \eqref{fggs}. This, however, does not impact the unique representability property.} of  $X$.  It follows from Lemma \ref{lem:Xseq_Sgap} that the only ergodic invariant probability measures  that put full measure on $X_{\rm lim}=X\setminus X_{{\rm seq}}$ are the Dirac measures $\delta_i, i\in {\rm I}$ supported on the fixed point $i^\infty$. Suppose that $\phi$ is a continuous potential with $P_{\rm top}(\phi)=\int \phi \,d\delta_i$. Let $(s_i^\ell)_\ell\in S_i$ where $s_i^\ell\to \infty$ as $\ell\to \infty$. Furthermore, for $j\in \{0,\dots,d-1\}$ with $j\not=i$, fix $s_j\in S_j$. Let $\sigma\in \Pi$ and $g_\ell
\in \cG$ be given by $s_0,\dots, s_i^\ell,\dots, s_{d-1}$ and $\sigma$ as in Equation \eqref{fggs}.
We denote the periodic measure supported of the periodic orbit $g_\ell^\infty$ by $\delta_{\ell}$. By using the continuity of $\phi$ and the structure of $g_\ell$, it is straight-forward to verify that $\int \phi \,d\delta_\ell\rightarrow \int \phi \,d\delta_i$ as $\ell \to \infty$. Moreover, it is easy to see that $d({\rm supp}\,\delta_{\ell} ,\cup_{i\in {\rm l} } {\rm supp}\,\delta_i)>0.$ 
Thus, Proposition \ref{proptriv} shows $FSSP(X)=C(X,\bR)$ for generalized gap shifts.  It is straight-forward to see that $\cG$ and $\cL(X)$ can be listed in order based on ordered listings of  the sets $S_0,\dots, S_{d-1}$ and permutation set $\Pi$.
Therefore, we apply Theorem A to establish the computability of  the topological pressure  on $X$. 
\end{proof}
\subsection{Beta-shifts}
Beta-shifts  are symbolic systems that were introduced by R\'{e}nyi in \cite{AR57}.  See \cite{P60,S2011} for subsequent  developments and additional references.  Beta-shifts have been studied both from the computability and the number theoretical point of view.  We briefly recall the basic properties of Beta-shifts following \cite{climenhaga2012intrinsic, cblog,S2011}.  We refer the reader to 
\cite{cblog,climenhaga2012intrinsic,Joh1999,S2011} for more details about the properties of Beta-shifts.

Let $\beta>1$ be a real number. 
The {\em Beta-shift} $X_{\beta}$ is the natural coding space associated with the $\beta$-transformation $T_{\beta}:[0,1)\rightarrow [0,1)$ given by
$T_{\beta}(x)=\beta x\,\, (\text{mod } 1)$. 
If $\beta\in\bZ$, then the corresponding Beta-shift is the full shift with $\beta$ symbols.  Therefore, we assume that $\beta$ is not an integer.

The coding space definition of  the Beta-shift follows: Given $x\geq 0$, let $\lfloor x\rfloor$ denote the integer part of $x$ and let $\{x\}$ denote the fractional part of $x$. For $x\in (0,1)$, we define 
two sequences  associated with $x$ as follows:
$a_1(x)=\{\beta x\}$ and $a_i(x)=\{\beta(a_{i-1}(x))\}$ for $i\geq 2$, and $x_1(x)=\lfloor \beta x\rfloor$ 
and $x_i(x)=\lfloor \beta a_{i-1}(x)\rfloor$ for $i\geq 2$. 
We call the sequence $(x_i)_{i\in\bN}$ the {\em beta-expansion} of $x$ where $x_i\in \mathcal{A}=\{0,\dots, \lfloor\beta\rfloor\}$.   
\begin{definition}
The Beta-shift $X_{\beta}$ is the closure of the set of $\beta$-expansions of all $x\in [0,1)$. 
\end{definition}
We recall that  $X_\beta$ is Sofic if and only if $\beta$ is eventually periodic.  In this case, it follows that if  the pre-periodic and the periodic parts of $\beta$ are given as input, then the computability of the topological pressure can be deduced from Proposition \ref{cpSofic}. Thus, we assume that $\beta$ is not eventually periodic.

An alternative and well-known method to define Beta-shifts is as  the set of paths on the countable directed labeled graph $\Gamma_\beta$, see Figure \ref{fig:beta} for an example. We briefly present this method,  following the exposition in \cite{climenhaga2012intrinsic}. We denote the lexicographic order on one-sided shift spaces  by $\preceq$. Let $\beta>1$ be a non-eventually periodic real number. There exists a unique sequence $b=b(\beta)=(b_1b_2b_3\cdots)$ which is the lexicographic supremum over all solutions to the equation
\[
\sum_{j=1}^\infty b_j\beta^{-j}=1.
\]
In other words, $b(\beta)$  is the lexicographic supremum of the $\beta$-expansions of $1$.  The Beta-shift is characterized by the following condition:
\begin{equation}\label{eq:shiftbetacondition}
x\in X_\beta \quad \Longleftrightarrow \quad f^n(x) \preceq b(\beta)\,\, {\rm for}\,\,{\rm all}\,\, n\in \bN_0
\end{equation}
Every Beta-shift can be presented as a countable directed labeled graph $\Gamma_\beta$ which is completely determined by $b=b(\beta)$.
We denote the vertices of  $\Gamma_\beta$ by $v_k$  for $k\in \bN$. For each $k\in \bN$, we add an edge from $v_k$ to $v_{k+1}$ and label it with $b_k$. 
Moreover, for all $k\in \bN$ and $i=0,\dots, b_k-1$, we add an edge  from $v_k$ to $v_1$ and label it with $i$. The Beta-shift $X_\beta$ coincides with the set of  sequences of labels  associated with the set of infinite paths on $\Gamma_\beta$. that start at $v_1$.
An example of such a graph $\Gamma_\beta$ is shown in Figure \ref{fig:beta}. 

\begin{figure}[hbt]
\begin{center}
\begin{tikzpicture}
\clip (-2.5,-1) rectangle (11.5,3);
\node [draw,circle,inner sep=0pt,minimum size=.6cm] (A) at (0,0) {$v_1$};
\node [draw,circle,inner sep=0pt,minimum size=.6cm] (B) at (2,0) {$v_2$};
\node [draw,circle,inner sep=0pt,minimum size=.6cm] (C) at (4,0) {$v_3$};
\node [draw,circle,inner sep=0pt,minimum size=.6cm] (D) at (6,0) {$v_4$};
\node [draw,circle,inner sep=0pt,minimum size=.6cm] (E) at (8,0) {$v_5$};
\filldraw (10.0,0) circle (.03);
\filldraw (10.5,0) circle (.03);
\filldraw (11,0) circle (.03);
\node at (9.8,0) (F) {};
\draw (E) edge[->] node[below]{$1$} (F);
\draw (A) edge[->] node[below]{$2$} (B);
\draw (B) edge[->] node[below]{$2$} (C);
\draw (C) edge[->] node[below]{$0$} (D);
\draw (D) edge[->] node[below]{$1$} (E);
\draw (B) edge[->,in=20,out=160] node[above]{0} (A);
\draw (B) edge[->,in=75,out=105] node[above]{1} (A);
\draw (D) edge[->,in=75,out=105] node[above]{0} (A);
\draw (E) edge[->,in=80,out=100] node[above]{0} (A);
\draw (A) edge[->,loop left,looseness = 15,in=160,out=200] node[left]{0} (A);
\draw (A) edge[->,loop left,looseness = 20,in=140,out=220] node[left]{1} (A);
\end{tikzpicture}
\end{center}
\caption{A graph representation of a Beta-shift.\label{fig:beta}}
\end{figure}
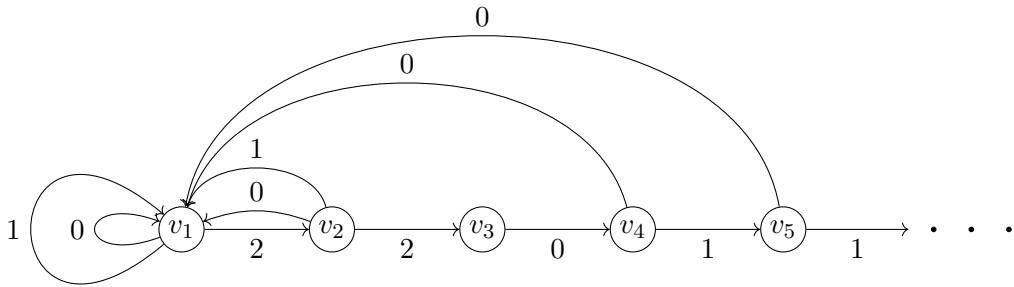

Since $X_\beta$ is a one-sided shift space, we can not deduce computability properties for the topological pressure by applying Theorem A directly.  We can, however, use the graph $\Gamma_\beta$  (which is irreducible) as a representation to associate to $X_\beta$ a  two-sided shift space $\hat{X}_\beta$.  A bi-infinite path  (or simply a path) $\gamma$ in $\Gamma_\beta$  is a bi-infinite sequence
of vertices $(v_{\ell_k})_{k\in \bZ}$ together with a bi-infinite sequence of edges $(e_{\ell_k})_{k\in \bZ}$ such that for all $k\in \bZ$, the edge $e_{\ell_k}$ is directed from $v_{\ell_k}$ to $v_{\ell_{k+1}}$. For a path $\gamma$, we define $x_\gamma=(x_k)_{k\in \bZ}$, where $x_k$ is the label associated the edge $e_{\ell_k}$. It follows that
\begin{equation*}
\hat{X}_\beta\eqdef\overline{\{x_\gamma: \gamma \,\,{\rm path}\,\, {\rm in}\,\, \Gamma_\beta\}}
\end{equation*}
is a two-sided shift space which has  the same language as $X_\beta$.
 As in the proof of Proposition \ref{lem:case1}, we associate a potential 
$\hat{\phi}\in C(\hat{X}_\beta,\bR)$ to $\phi\in C(X_\beta,\bR)$ such that $\phi(x)=\hat{\phi}(y)$ whenever $x_k=y_k$ for all $k\geq 0$. 
It follows from  $\cL(X_\beta)=\cL(\hat{X}_\beta)$ and the definition of the topological pressure that
\begin{equation}\label{eqivpres}
P_{\rm top}(X_\beta,\phi)=P_{\rm top}(\hat{X}_\beta,\hat{\phi})
\end{equation}
for all $\phi\in C(X,\bR)$. Therefore, the computability of the topological pressure on $\hat{X}_\beta$ guarantees the computability of the topological pressure on $X_\beta$.
We now present our computability results for Beta-shifts. We start with a preliminary result.

\begin{lemma}\label{uniquelyrepresentablebeta}
Let $\beta>1$ be a non-integral real number which is not eventually periodic. Then $\hat{X}_\beta$ is a coded shift with unique representation
\begin{equation*}
\mathcal{G_\beta}=\{g_{j,i}=b_1\cdots b_{j}i: i<b_{j+1}\}\cup \{g_i=i:i<b_1\}.
\end{equation*}
\end{lemma}
\begin{proof}
We observe that $\hat{X}_\beta$ is a coded shift since it is presented by the irreducible countable labeled directed graph $\Gamma_\beta$,
see  \cite{MarcusLind1995}.
Let $\gamma$ be a bi-infinite path on $\Gamma_\beta$.  We observe that, for any finite subpath $w$ of $\gamma$, $w$ can be extended to a finite path $\widetilde{w}$ that begins and ends at vertex $v_1$.  Such a path can be written as a finite concatenation of generators in $\mathcal{G}_\beta$.  Therefore, $0^\infty\widetilde{w}0^\infty\in X_{\rm seq}$, and, by using increasing-length subpaths of $\gamma$, we can construct a sequence of elements of $X_{\rm seq}$ which converge to $x_\gamma$, and we conclude that $\cG_\beta$ is a generating set for $\hat{X}_\beta$.

Next we show that $\cG_\beta$ is a unique representation.
We call generators of the form $g_{j,i}$ {\em forward generators} where $b_1\cdots b_j$ forms the {\em forward step} and the final $i$ is the {\em backward step}.  Similarly, we call generators of the form $g_i$ {\em backward generators}.  We observe that the character $b_1$ can only appear in forward generators.

Let $x=\cdots p_{-2}p_{-1}p_0p_1p_2\cdots=\cdots q_{-2}q_{-1}q_0q_1q_2\cdots$ be two representations of $x$ by a concatenation of generators.  We first show that the forward generators, i.e., those generators containing $b_1$'s, are identical.  Suppose that two forward generators $p_i$ and $q_j$ overlap.  Without loss of generality, we assume that $p_i$ starts at or before $q_j$.  Let $k_p$ be the starting point of $p_i$ and $k_q$ be the starting point of $q_j$.  We show that $p_i$ cannot end before $q_j$ ends as follows: Suppose that $p_i$ ends at index $l$ and that $q_j$ continues beyond this index, then, in the overlap, $b_1=x_{k_p}=b_{k_q-k_p+1},\dots,b_{l-1}=x_{l-1}=b_{k_q-k_p+l-1}$ and $b_{l}=x_l<b_{k_q-k_p+l}$.  This contradicts the characterization of Beta-shifts in Equation \eqref{eq:shiftbetacondition}.  If $p_i$ and $q_j$ begin at the same point, symmetry shows that $p_i=q_j$.  Otherwise, since $p_i$ begins with $b_1$, there is some $q_{j'}$ which overlaps with the beginning of $p_i$.  By the argument above, $q_{j'}$ cannot end before $p_i$ ends, but this contradicts the existence of $q_j$.  Therefore, $p_i$ and $q_j$ begin at the same index and are equal.  

Since the forward generators of the two expansions of $x$ have been shown to be equal, all that is left are the gaps between forward generators.  Since $b_1$ doesn't appear in these gaps, all of these remaining generators must be backward generators.  For any such index $k$, $p_i=x_k=q_j$ as backward generators correspond to single characters.  Therefore, the two expansions of $x$ are identical and the result follows.
\end{proof}

For two generators $g_{j_1,i_1}$ and $g_{j_2,i_2}$, we say that $g_{j_1,i_1}$ appears as a substring in the {\em interior} of $g_{j_2,i_2}$ if $g_{j_1,i_1}$ is a substring of $b_2\dots b_{j_2}$.  An {\em infinite chain} of generators in $\cG_\beta$ is defined to be a sequence of generators $(g_{j_k,i_k})_{k\in\bN}$ where each $g_{j_k,i_k}$ appears as a substring in the interior of $g_{j_{k+1},i_{k+1}}$.

\begin{theorem}\label{comppresbeta}Let $\beta>1$ be a non-integral real number which is not eventually periodic.  In addition, assume that $\cG_\beta$ does not have any infinite chains of generators.  Then $FSSP(\hat{X}_{\beta})=C(\hat{X}_{\beta},R)$.
Moreover, if $\beta$ is computable, then the topological pressure is computable
on $X_\beta$.
\end{theorem}

\begin{proof} 
We show that no invariant measure puts full measure on $\hat{X}_{\beta}\setminus \hat{X}_{\beta,\rm seq}$.  We write this difference as a union of two zero-measure sets $A$ and $B$.

Let $A$ be the set of $x\in\hat{X}_{\beta}\setminus \hat{X}_{\beta,\rm seq}$ such that $x$ has a right-tail of the form $b_1b_2b_3\dots$.  The starting index of such a tail is unique since otherwise $\beta$ is periodic.  Let $A_j$ be the set of all $x\in A$ which start at index $j$.  We observe that $f(A_j)=A_{j-1}$ and that $A=\bigcup_{j\in\bZ}A_j$.  Since the $A_j$'s are disjoint, for any $\mu\in\cM$, $\mu(A_0)=0$ since otherwise $\mu(A)$ is infinite.  Therefore, $\mu(A)=0$.

Let $B$ be the set of $x\in\hat{X}_\beta\setminus\hat{X}_{\beta,\rm seq}$ such that there exists a index $k$ such that for infinitely many $l\in\bN$, $x_{k-l}\dots x_k$ is a forward generator in $\cG_\beta$.  We show that only one such index exists as follows: Suppose that $k_p$ and $k_q$ are two indices with the described property, and, without loss of generality, suppose that $k_p<k_q$.  Fix $l_q\in\bN$ such that $x_{k_q-l_q}\dots x_{k_q}$ is a forward generator.  Then, we may choose an $l_p\in\bN$ so that $k_p-l_p<k_q-l_q$ and $x_{k_p-l_p}\dots x_{k_p}$ is a forward generator.  Since generators begin with $b_1$, the argument in Lemma \ref{uniquelyrepresentablebeta} shows that this is a contradiction.  Let $B_j$ be the set of all $x\in B$ which start at index $j$.  By the same argument as for $A$, we find that for all $\mu\in\cM$, $\mu(B)=0$.

We now show that $A\cup B=\hat{X}_{\beta}\setminus \hat{X}_{\beta,\rm seq}$.  The situation not covered by $A$ and $B$ would be where there are infinitely many indices $k$, each of which has finitely many $l$'s so that $x_{k-l}\dots x_k$ is a forward generator.  In addition, there are sequences $(k_i)_{i\in\bN}$ and $(l_i)_{i\in\bN}$ such that $k_i\rightarrow\infty$ and $(k_i-l_i)\rightarrow-\infty$.  In this situation, $\cG_\beta$ must contain infinite chains of generators, which is not possible.

It is shown in \cite{S2011} that the language of the  Beta-shift is recursive iff $\beta$ is a computable real number. Moreover, since $\beta$  is a computable non-preperiodic real number, there exists a Turing machine that implements the greedy algorithm to compute $b(\beta)$. Thus there is a Turning machine that generates $\cL(\hat{X}_\beta)$ and 
$\cG_\beta$. Applying Theorem A shows that $P_{\rm top}(\hat{X}_\beta,.)$ is computable on $C(\hat{X}_{\beta},R)$. Finally, the topological pressure  is computable by Equality \eqref{eqivpres}
on $X_\beta$.
\end{proof}
We remark that the computability result in Theorem \ref{comppresbeta} is only novel for non-constant potentials $\phi$. Indeed, if $\phi=c$ then $P_{\rm top}(X_\beta,\phi)=h_{\rm top}(X_\beta)+c=\log \beta +c$ which is computable since $\beta$ is computable.

\section{Concluding remarks} \label{sec:concl}

We end with some remarks on our assumptions and results as well as some open questions. 
\subsection{Examples of noncomputability of pressure and entropy}
The following examples shows that in general $FSP(X)\not=C(X,\bR)$, and, hence, $FSSP(X)\not=C(X,\bR)$.
\begin{example}\label{counter}
Let $X$ be the coded shift generated by $\mathcal{G}=\{(000)^k1^k\}_{k\in\bN}$.  Let $\phi\in C(\Sigma,\bR)$ where 
$$\phi(x)=\begin{cases} 1 &\mbox{if } x_0=1   \\
0 & \mbox{if } x_0=0.\end{cases}$$
Let $\delta_{\mathcal{O}(1)}$ be the Dirac measure supported on the periodic point $\mathcal{O}(1)$.  
We observe that, first, $\htop(X)\leq \log 2$; second, by the Birkhoff ergodic theorem, for any invariant measure $\mu\not=\delta_{\mathcal{O}(1)}$, $\int \phi d\mu<\frac{1}{4}$; and, third, $\int \phi d\delta_{\mathcal{O}(1)}=1$.  Hence, it follows from the variational principle in Equation \eqref{varpri}, that $\Ptop(X,\phi|_X)=1$.  Moreover, $\phi\not\in FSP(X)$, so $FSP(X)\not=C(X,\bR)$.  Furthermore, the pressure map $\Ptop: \ClosedShift\times C(\Sigma, \mathbb{R})\rightarrow\mathbb{R}$ is not computable at $(X,\phi)$ since, for any $m\in\bN$, let $X_m$ be the coded shift generated by $\mathcal{G}_m=\{(000)^k1^k\}_{k\in\{1,\dots,m\}}$.  Then, there is a uniform gap between the pressures $P(X_m,\phi)$ and $P(X,\phi)$ since $P(X_m,\phi)\leq\frac{1}{4}+\log 2<1$.

Additionally, we observe that
$$\Ptop(\phi,X)=\sup_{\mu\in \mathcal{M}(X)}\int \phi d\mu.$$
Moreover, this equality holds for all small perturbations of $\phi$.  Hence, for shifts with positive entropy, the set of potentials where Inequality (\ref{eqposent}) holds may not be dense.
\end{example}
  
The following example shows that the entropy may not be computable when the constant zero function is not in $FSP(X)$.  We use a graph and construction inspired by \cite{Petersen1986}, see Figure~\ref{fig:graph:Peterson}.  

\begin{example}
Let $G$ be the graph in Figure~\ref{fig:graph:Peterson}.  Suppose that $\phi=0$, i.e., the pressure is the entropy and the potential is a computable function.  Let $X_0\subset\{0,1\}^{\mathbb{Z}}$ be a minimal set which is not intrinsically ergodic and where $\htop(f)>0$, see \cite[p.\ 157]{ErgTh_compact2006} for additional details.  We recall that a minimal set cannot contain an SFT with positive entropy because SFTs include periodic points.  Let $(m_0,m_1,\dots)$ be the right-half of an element $m$ of $X_0$.  We label the horizontal arrows in $G$ by the $m_i$'s and the curved edges in $G$ by $2$.

\begin{figure}[hbt]
\begin{center}
\begin{tikzpicture}
\node at (0,0) (A) {};
\node at (2,0) (B) {};
\node at (4,0) (C) {};
\node at (6,0) (D) {};
\node at (8,0) (E) {};
\filldraw (A) circle (.1);
\filldraw (B) circle (.1);
\filldraw (C) circle (.1);
\filldraw (D) circle (.1);
\filldraw (E) circle (.1);
\filldraw (9.0,0) circle (.03);
\filldraw (9.5,0) circle (.03);
\filldraw (10,0) circle (.03);
\node at (8.8,0) (F) {};
\draw (E) edge[->] (F);
\draw (A) edge[->] node[below]{$m_0$} (B);
\draw (B) edge[->] node[below]{$m_1$} (C);
\draw (C) edge[->] node[below]{$m_2$} (D);
\draw (D) edge[->] node[below]{$m_3$} (E);
\draw (B) edge[->,in=20,out=160] node[above right]{2} (A);
\draw (C) edge[->,in=33.33,out=146.67] node[above right]{2} (A);
\draw (D) edge[->,in=46.67,out=133.33] node[above right]{2} (A);
\draw (E) edge[->,in=60,out=120] node[above right]{2} (A);
\end{tikzpicture}
\end{center}
\caption{Labeled graph $G$ whose underlying graph appears in \cite{Petersen1986}.  The horizontal edges are labeled with a minimal subset of $\{0,1\}^{\bZ}$ which is not intrinsically ergodic and with positive topological entropy.\label{fig:graph:Peterson}}
\end{figure}
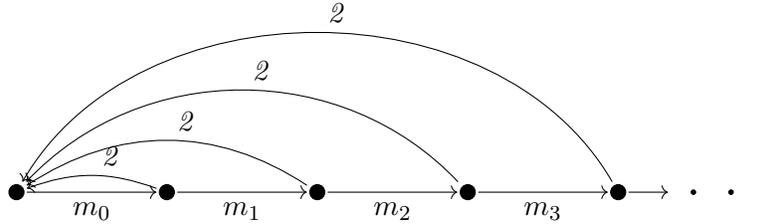

Let $L(G)$ be the set of all labelings of bi-infinite paths $\xi=\cdots e_{-1}e_0e_1\cdots$ on $G$.  This graph determines a $2$-block map $\pi: L(G)\rightarrow \{0,1,2\}^{\pm}$.  Let $X$ denote the closure of $\pi(L(G))$. Then $X_0\subset X$, so $\htop(f,X)\geq\htop(f,X_0).$  For any integer $k$, we define $N(k)$ so that $m_0m_1\cdots m_{N(k)}$ contains all the length-$k$ strings in $m$.  Let $G_k$ be the subgraph of $G$ formed by deleting all vertices to the right of endpoint of $m_{N(k)}$.  The language of $G_k$ agrees with the language of $G$ up to length $k$, so the corresponding shift $X_k$ satisfies $d(X,X_k)<2^{-k}$.  On the other hand, $\htop(X_k)=0$ while $\htop(X)>0$.  Hence the zero function is not in $FSP(X)$. 
\end{example}

Beta transformations are prototypes of piecewise expanding interval transformations. Our work leads to the problem of identifying which types of  piecewise expanding interval maps satisfy $FSP(X)=C(X,\bR)$.  As first step in this direction one could consider Alpha-Beta transformations, see, e.g. \cite{cblog}.
\begin{question}Does $FSP(\hat{X}_{\alpha,\beta})= C(\hat{X}_{\alpha,\beta},R)$ hold for all Alpha-Beta-shifts $\hat{X}_{\alpha,\beta}$? 
\end{question}

\subsection{Computability of the integral function}
The computability of the integral function 
$$I: \Sigma_{\rm invariant}\times C(\Sigma,\mathbb{R})\rightarrow \mathbb{R}\quad\text{where}\quad (X,\phi)\mapsto \sup_{\mu\in\mathcal{M}(X)}\int \phi d\mu$$
is closely related to the computability of the pressure.  Indeed, the non-computability in Example~\ref{counter} is derived from the non-computability of this function.  The computability of this integral has been studied for SFTs in \cite{BSW}.  Due to the strong relationship to the computability of the pressure, we pose the question:
\begin{question}
On which subset of $\Sigma_{\rm invariant}\times C(\Sigma,\mathbb{R})$ is the function $I$ computable?
\end{question}

\section*{Acknowledgement}The authors thank Ethan Akin and Benjamin Steinberg for helpful discussions about  uniquely representable  coded shifts.


\bibliographystyle{abbrv}
\bibliography{Shuddho_references,BSW_References}

\end{document}